\documentclass[11pt,a4paper,reqno]{amsart}
\usepackage{amsmath,amssymb,amsfonts,epsfig,yhmath}

\newtheorem{theorem}{Theorem}[section]
\newtheorem{lemma}[theorem]{Lemma}
\newtheorem{prop}[theorem]{Proposition}
\newtheorem{definition}[theorem]{Definition}
\newtheorem{corollary}[theorem]{Corollary}
\newtheorem{remark}[theorem]{Remark}

\def\step#1#2{\par\medskip \noindent{\underline{\it Step~#1.}}\emph{ #2}\par}
\def\eps{\varepsilon}
\def\lbar{\bar{\ell}}

\def\bal{\begin{aligned}}
\def\eal{\end{aligned}}
\def\proofof#1{\begin{proof}[Proof (of~#1).]}

\def\XXint#1#2#3{{\setbox0=\hbox{$#1{#2#3}{\int}$} \vcenter{\vspace{-1pt}\hbox{$#2#3$}}\kern-.5\wd0}}
\def\Xint#1{\mathchoice {\XXint\displaystyle\textstyle{#1}}{\XXint\textstyle\scriptstyle{#1}}{\XXint\scriptstyle\scriptscriptstyle{#1}}{\XXint\scriptscriptstyle\scriptscriptstyle{#1}}\!\int}
\def\intmed{\hbox{\ }\Xint{\hbox{\vrule height -0pt width 10pt depth 1pt}}}

\numberwithin{equation}{section}

\def\R{\mathbb R}
\def\N{\mathbb N}
\def\S{\mathbb S}
\def\H{{\mathcal H}}

\def\angle#1#2#3{#1\widehat{#2}#3}
\def\arc#1{\wideparen{#1}}
\def\step#1#2{\par\noindent{\underline{\it Step~#1.}}\emph{ #2}\\}
\def\segm#1{\overline{#1}}

\title{On the piecewise approximation of bi-Lipschitz curves}

\author[A.~Pratelli]{Aldo Pratelli}
\address{Department of Mathematics, University of Erlangen, Cauerstrasse 11, 90158 Erlangen, Germany}
\email{pratelli@math.fau.de}

\author[E.~Radici]{Emanuela Radici}
\address{Department of Mathematics, University of Erlangen, Cauerstrasse 11, 90158 Erlangen, Germany}
\email{radici@math.fau.de}

\begin{document}
\begin{abstract}
In this paper we deal with the task of uniformly approximating an $L$-biLipschitz curve by means of piecewise linear ones. This is rather simple if one is satisfied to have approximating functions which are $L'$-biLipschitz, for instance this was already done with $L'=4L$ in~\cite[Lemma~5.5]{DP}. The main result of this paper is to do the same with $L'=L+\eps$ (which is of course the best possible result); in the end, we generalize the result to the case of closed curves.
\end{abstract}

\maketitle

\section{Introduction}

A problem raised by Evans and Ball in the 1980's (see for instance~\cite{Ball2}), and still open in its full generality, is the following: can one approximate a planar bi-Sobolev homeomorphism with diffeomorphisms, or piecewise affine homeomorphisms, in the bi-Sobolev sense? This would be rather important for applications in the context of the non-linear elasticity. This problem and its partial solutions have an interesting history, one can see for instance the papers~\cite{BMC,IwKovOnn,DP,HP} to have an overview of what is now known.\par

In all the available results, for instance in those cited above, the authors use quite different strategies, but a common ingredient is to divide the domain in simple ones, namely, triangles or squares, and then to work on each of them. And, in particular, it is often important to approximate the value of a homeomorphism on the boundary of the triangle or square. In other words, the much simpler one-dimensional task (i.e., approximating a function defined on a segment, or on the boundary of a square) is one of the ingredients to solve the two-dimensional one, actually usually a very easy ingredient. Let us state this more precisely: we have a function $\varphi:[0,1]\to \R^2$, and we look for a piecewise linear function $\varphi_\eps:[0,1]\to\R^2$, which is uniformly close to $\varphi$ and which coincides with $\varphi$ at $0$ and $1$; this is of course very easy to reach. In addition, one has often the information that $\varphi$ is $L$-biLipschitz; in this case, it would be also interesting to have an estimate on the bi-Lipschitz constant of $\varphi_\eps$ (which is surely bi-Lipschitz, since it is piecewise linear). Surprisingly enough, this does not come for free; in particular, in~\cite[Lemma~5.5]{DP} it was proved that one can obtain a $4L$-biLipschitz approximating function $\varphi_\eps$, and the proof was simple but not straightforward.\par

The goal of the present paper is to obtain the sharp result in this direction, namely, that it is possible to have an approximating function $\varphi_\eps$ which is $(L+\eps)$-biLipschitz. To state it, let us first recall the definition of the bi-Lipschitz property.

\begin{definition}
The function $f: [0,1]\to \R^2$ is said to be \emph{$L$-biLipschitz} if for every $p,\, q\in [0,1]$
\[
\frac 1 L\, |p-q|\leq |f(p)-f(q)| \leq L |p-q|\,.
\]
Notice that the second inequality is the usual $L$-Lipschitz property; through this paper, we will refer to the first inequality as the \emph{inverse $L$-Lipschitz property}.
\end{definition}

Notice that a function can be $L$-biLipschitz only if $L\geq 1$, and actually if $L=1$ then the function must be linear (and then, there is no need of approximating it with piecewise linear functions!). As a consequence, we can always think that $L>1$. We can now state our main result.

\begin{theorem}\label{main}
Let $\varphi:[0,1]\to\R^2$ be an $L$-biLipschitz function, and $\eps>0$. Then there exists an $(L+\eps)$-biLipschitz function $\varphi_\eps:[0,1]\to\R^2$ such that
\begin{align}\label{claimmain}
\varphi_\eps(0)=\varphi(0)\,, && \varphi_\eps(1)=\varphi(1)\,, && \|\varphi-\varphi_\eps\|_{L^\infty}\leq \eps\,,
\end{align}
and $\varphi_\eps$ is finitely piecewise linear on $[0,1]$.
\end{theorem}

The plain of the paper is very simple. Sections~\ref{sect2} and~\ref{sect3} contain the main ingredients of the proof, namely, the study of how one can approximate a function near Lebesgue points for $\varphi'$, and the study of how to treat the remaining small intervals. These two ingredients will be then put together in Section~\ref{sect4}, which contains the proof of Theorem~\ref{main}. Finally, Section~\ref{sect5} is devoted to generalize the result to the case of closed curves, that is, instead of functions defined on $[0,1]$ we consider functions defined on $\S^1$; this is obtained in Theorem~\ref{main2}.

\subsection{Notation}

In this paper we will use very little notation. We list the main things here for the sake of clarity. The length-measure is denoted by $\H^1$. Given two points $x,\, y\in \R^2$, we denote by $xy$ the segment joining them; depending on the situation, for the ease of notation we denote its length either by $|y-x|$, or by the quicker symbol $\segm{xy}$, or also by $\H^1(xy)$. Once a curve $\gamma:[a,b]\to\R^2$ is fixed, the arc between two points $P$ and $Q$ is denoted by $\arc{PQ}$, and its length is then $\H^1(\arc{PQ})$. In particular, for any $a\leq x<y\leq b$, $\arc{\gamma(x)\gamma(y)}$ is the arc connecting $\gamma(x)$ to $\gamma(y)$. Finally, given any three points $x,\,y,\, z\in\R^2$, we write $\angle xyz\in [0,\pi]$ to denote the angle between the segments $xy$ and $yz$.

\section{The ``Lebesgue intervals''\label{sect2}}

In this section we show that, on a small interval around a Lebesgue point for $\varphi'$, it is possible to replace the function $\varphi$ with a linear one. Since Rademacher Theorem ensures that almost every point of $[0,1]$ is a Lebesgue point for $\varphi'$, being $\varphi$ a Lipschitz function, we will be eventually able to repeat this argument on a large number of non-intersecting intervals which fill a big portion of $[0,1]$. In the end, we can prove the following result.

\begin{prop}\label{step1}
Let $\varphi:[0,1]\to \R^2$ be an $L$-biLipschitz function, and let $\eps>0$. Then there exists an $(L+\eps)$-biLipschitz function $\varphi_\eps:[0,1]\to\R^2$ such that~(\ref{claimmain}) holds true, and $\varphi_\eps$ is finitely piecewise linear on a finite union of intervals $A\subseteq [0,1]$ such that $|[0,1]\setminus A|\leq \eps$.
\end{prop}

As we said above, the main brick to prove this result concerns the modification of $\varphi$ on a single small interval. Before stating it, we need the following piece of notation.

\begin{definition}\label{varphist}
Let $\varphi: [0,C]\to \R^2$ be a function, and let $0\leq s < t \leq C$. We set $\varphi_{st}:[0,C]\to \R^2$ the function defined as
\[
\varphi_{st}(x):= \left\{\begin{array}{ll}
\varphi(x) & \hbox{if $x\notin (s,t)$}\,,\\
\bal\varphi(s) + \frac{x-s}{t-s} \, \big(\varphi(t)-\varphi(s)\big)\eal\qquad & \hbox{if $x\in (s,t)$}\,.
\end{array}\right.
\]
Moreover, we call $t^+=s+|\varphi(t)-\varphi(s)|/L$ and $\varphi_{st}^+:[0,C-(t-t^+)]\to \R^2$ the function
\[
\varphi_{st}^+(x):= \left\{\begin{array}{ll}
\varphi(x) & \hbox{if $x\leq s$}\,,\\
\bal\varphi(s) + \frac{L(x-s)}{|\varphi(t)-\varphi(s)|} \, \big(\varphi(t)-\varphi(s)\big)\eal\qquad & \hbox{if $s<x<t^+$}\,,\\
\varphi(x+t-t^+) & \hbox{if $t^+\leq x\leq C-(t-t^+)$}\,.
\end{array}\right.
\]
\end{definition}
In words, the function $\varphi_{st}$ coincides with $\varphi$ out of the interval $(s,t)$, while the curve $\varphi$ in $(s,t)$ is replaced by the segment connecting $\varphi(s)$ to $\varphi(t)$. The function $\varphi_{st}^+$ behaves in the very same way, except that the segment is parametrized at the (maximal possible) speed $L$.

\begin{lemma}\label{Lebesgue}
Let $\varphi:[0,1]\to \R^2$ be an $L$-biLipschitz function, and let $\eps>0$ be small enough. For any $x\in (0,1)$ which is a Lebesgue point for $\varphi'$, there exists $\bar\ell=\bar\ell(x)>0$ such that, for any $\ell\leq\bar\ell$, there is a sets $I_\ell(x)\subseteq (x-\ell,x+\ell)$ with 
\begin{align}\label{lungIpm}
\big|I_\ell(x)\big|  \geq (2-\eps) \ell\,,
\end{align}
so that for every $s<t\in I_\ell(x)$ the function $\varphi_{st}$ is $(L+\eps)$-biLipschitz and satisfies $\|\varphi-\varphi_{st}\|_{L^\infty}<\eps$. Moreover, for every $s \in I_\ell(x)$ and $t_1,\, t_2 \in (x-\ell/\eps,x+\ell/\eps)$ with $t_1<s<t_2$, the directions of the segments $\varphi(t_1)\varphi(s)$ and $\varphi(s)\varphi(t_2)$ coincide up to an error $2\eps$.
\end{lemma}

In this section we will first show Lemma~\ref{Lebesgue}, and then use it as a tool to obtain Proposition~\ref{step1}.

\proofof{Lemma~\ref{Lebesgue}}
We divide the proof in three steps for the sake of clarity.
\step{I}{The $L$-Lipschitz property of $\varphi_{st}$ and a uniform estimate for $\varphi-\varphi_{st}$.}
In this step we show that $\varphi_{st}$ is $L$-Lipschitz for any choice of $s,t$ in $[0,1]$, and we give an estimate for $\|\varphi-\varphi_{st}\|_{L^\infty}$. Let us fix $0\leq s\leq t\leq 1$ and take two arbitrary points $y\neq z\in [0,1]$: we have to check that
\[
\frac{|\varphi_{st}(y)-\varphi_{st}(z)|}{|y-z|} \leq L\,.
\]
This is clearly true if both $y,\,z \in [0,1]\setminus (s,t)$, since in this case $\varphi_{st}=\varphi$ at both $y$ and $z$; on the other hand, if both $y,\, z \in [s,t]$, then by definition
\begin{equation}\label{bothsides}
\frac{|\varphi_{st}(y)-\varphi_{st}(z)|}{|y-z|} 
=\frac{|\varphi_{st}(s)-\varphi_{st}(t)|}{|s-t|} 
=\frac{|\varphi(s)-\varphi(t)|}{|s-t|} \leq L\,.
\end{equation}
Finally, let us suppose that one of the points $y$ and $z$ belongs to $(s,t)$, and the other one to $[0,1]\setminus [s,t]$; by symmetry, we can assume $s<y<t<z$. Thus, by the above observations and by the triangular inequality we have
\[
|\varphi_{st}(y)-\varphi_{st}(z)| \leq
|\varphi_{st}(y)-\varphi_{st}(t)| + |\varphi_{st}(t)-\varphi_{st}(z)|
\leq L |y-t| + L |t-z| = L |y-z|\,.
\]
Concerning the uniform estimate, it says that
\begin{equation}\label{stimaLinfinito}
\|\varphi-\varphi_{st}\|_{L^\infty}\leq 2L |t-s|\,.
\end{equation}
Indeed, calling for brevity $d=|t-s|$, for any $x\in [s,t]$ one has $|\varphi(x)-\varphi(s) | \leq L |x-s| \leq L d$. As an immediate consequence, for any $y\in [s,t]$ we also get $|\varphi_{st}(y) - \varphi(s)| \leq Ld$, hence in turn $\|\varphi_{st}-\varphi\|_{L^\infty}\leq 2Ld$, that is, (\ref{stimaLinfinito}).

\step{II}{Definition of $\ell(x)$ and $I_\ell(x)$ and the estimate on the directions.}
Let $x$ be a Lebesgue point for $\varphi'$. By definition, for every $\delta>0$ there exists a strictly positive constant $\bar h=\bar h(x)<1/(4L)$ such that, for any $h<\bar h$,
\begin{equation}\label{Lebcond}
\intmed_{x-h}^{x+h} |\varphi'(z) - \varphi'(x)|dz < \delta\,.
\end{equation}
Let us now assume for simplicity that $\varphi'(x)$ is an horizontal vector, and for any $p<q$ in $(x-h,x+h)$, let us call $\tau_{pq}\in\S^1$ the direction of the segment $\varphi(p)\varphi(q)$. Since $|\varphi'(x)|\geq 1/L$, we immediately obtain that for any interval $(p,q)\subseteq (x-h,x+h)$ the following holds,
\begin{align}\label{predefofA}
M(p,q):=\intmed_p^q |\varphi'(z)-\varphi'(x)|\, dz < \frac \eps {2L} && \Longrightarrow && |\tau_{pq}| < \eps\,.
\end{align}
We want now to find a particular set $A\subseteq (x-3h,x+3h)$ such that
\begin{equation}\label{defofA}
M(p,q)< \frac \eps{2L} \qquad \forall\, p,\, q \in (x-h,x+h):\, (p,q)\notin A\times A\,.
\end{equation}
Notice that we are asking $M(p,q)$ to be small as soon as \emph{at least one} between $p$ and $q$ belongs to $(x-h,x+h)\setminus A$. To define $A$, let us start simply by letting $A=\emptyset$ if $M(p,q)<\eps/(2L)$ is true for every pair $p,\,q\in (x-h,x+h)$, so that~(\ref{defofA}) trivially holds.\par

Otherwise, let $p_1< q_1\in (x-h,x+h)$ be two points maximizing $\H^1(pq)$ among all the pairs for which $M(p,q)\geq \eps/(2L)$: notice that this is possible by the fact that $\varphi'$ is an $L^1$ function on the compact interval $[x-h,x+h]$. Then, let us define $I_1 =(p_1^-,q_1^+)$, being
\begin{align}\label{def-+}
p_1^- = p_1 - (q_1 - p_1) \,, && q_1^+ = q_1 + (q_1 - p_1)\,.
\end{align}
Notice that by construction $I_1\subseteq (x-3h,x+3h)$. Now, if~(\ref{defofA}) is satisfied with $A=I_1$ we stop here, otherwise let $p_2<q_2\in (x-h,x+h)$ be two points maximizing $\H^1(pq\setminus I_1)$ among the pairs for which $M(p,q)\geq \eps/(2L)$, and let $I_2=(p_2^-,q_2^+)$ where $p_2^-$ and $q_2^+$ are defined as in~(\ref{def-+}). Notice that, by definition, it is possible that $p_2$ or $q_2$ belong to $I_1$, but the intervals $(p_1,q_1)$ and $(p_2,q_2)$ are surely disjoint. Indeed, by the maximality in the definition of $p_1$ and $q_1$ we have that $\H^1(p_2q_2)\leq \H^1(p_1q_1)$; as a consequence, the intervals $(p_1,q_1)$ and $(p_2,q_2)$ could intersect only if both $p_2$ and $q_2$ belong to $I_1$: but then, $\H^1(p_2q_2\setminus I_1)=0$, against the maximality in the definition of $p_2$ and $q_2$. Moreover, as before, $I_2\subseteq (x-3h,x+3h)$. We continue our definition of the intervals $I_j$ recursively, being at any step $p_j<q_j\in (x-h,x+h)$ two points maximizing $\H^1\big(pq\setminus \cup_{i=1}^{j-1} I_i\big)$ among the pairs for which $M(p,q)\geq \eps/(2L)$, noticing that the different intervals $(p_j,q_j)$ are disjoint, and stopping the construction if $A=\cup_{i=1}^j I_i$ satisfies~(\ref{defofA}).\par

Thus, either we stop after finitely many steps, and this means that~(\ref{defofA}) holds true being $A$ a finite union of intervals, or we end up with a sequence of intervals $I_j=(p_j^-,q_j^+),\, j\in\N$. Since all the different ``internal intervals'' $(p_j,q_j)$ are disjoint, the sum of the lengths is bounded, hence $|I_j|\to 0$ when $j\to \infty$. As a consequence, we can easily check that~(\ref{defofA}) holds true by setting
\[
A:= \Bigg\{ z\in (x-3h,x+3h):\, \liminf_{\nu\to 0} \frac{|(z-\nu,z+\nu)\cap \cup_j I_j |}{2\nu}> 0\Bigg\}\,,
\]
that is, $A$ is the set of points having strictly positive density with respect to $\cup_j I_j$. To do so, let us assume the existence of $p<q \in (x-h,x+h)$ such that at least one between $p$ and $q$ does not belong to $A$, but $M(p,q)\geq \eps/(2L)$. We can immediately notice that
\begin{equation}\label{nullmeasure}
\H^1\big(pq \setminus \cup_j I_j \big)=0\,:
\end{equation}
indeed, if the above measure were some quantity $\xi>0$, then the fact that the interval $(p,q)$ was not chosen at the $j$-th step gives that
\[
\H^1\big(p_jq_j\setminus \cup_{i=1}^{j-1} I_i\big) \geq 
\H^1\big(pq\setminus \cup_{i=1}^{j-1} I_i\big) \geq
\H^1\big(pq\setminus \cup_{i\in\N} I_i\big)=\xi\,,
\]
hence in particular $|I_j|\geq \xi$ for every $j$, while we have already noticed that $|I_j|\to 0$. On the other hand, (\ref{nullmeasure}) implies that both $p$ and $q$ have at least density $1/2$ for $\cup_j I_j$, so they both belong to $A$, against the assumption. Hence, the validity of~(\ref{defofA}) has been established.\par

We define then $\ell=\tilde\eps h$ for some $\tilde\eps=\tilde\eps(\eps,L) < \eps$ to be specified later, and we set
\[
I_\ell(x) = (x-\ell,x+\ell) \setminus A\,.
\]
Keep in mind that, since $h$ is any positive constant smaller than $\bar h(x)$, then also $\ell$ can be chosen as any positive constant smaller than $\bar\ell(x)=\tilde \eps \bar h(x)$. To conclude this step, we give an estimate of the length of $A$, namely,
\[\begin{split}
|A|&\leq \sum_{j=1}^{+\infty}    |I_j|
= 3 \sum_{j=1}^{+\infty}    \H^1(p_jq_j)
\leq \frac{6L}\eps\,  \sum_{j=1}^{+\infty} \int_{p_j}^{q_j} |\varphi'(z)-\varphi'(x)| \,d z\\
&\leq  \frac{6L}\eps\, \int_{x-h}^{x+h} |\varphi'(z)-\varphi'(x)| \, d z
<\frac{12Lh \delta}\eps <h\eps\tilde\eps =\ell\eps\,,
\end{split}\]
where we have used the definition of $A$, the fact that $M(p_j,q_j)\geq \eps/(2L)$ for every $j\in\N$, the fact that all the intervals $(p_j,q_j)$ are disjoint, and~\eqref{Lebcond}, and where the last inequality holds true as soon as $\delta\leq \eps^2\tilde\eps/(12L)$. As a consequence, the validity of~(\ref{lungIpm}) follows.\par

To conclude this step, we take two points $s\in I_\ell(x)$ and $t_1,\,t_2\in (x-h,x+h)$ with $t_1<s<t_2$. Applying~(\ref{defofA}) at both pairs $(t_1,s)$ and $(s,t_2)$, and keeping in mind~(\ref{predefofA}), we get that both the segments $\varphi(t_1)\varphi(s)$ and $\varphi(s)\varphi(t_2)$ are horizontal up to an error $\eps$, thus in turn the two directions coincide up to an error $2\eps$. Notice that, since $(x-\ell/\eps,x+\ell/\eps)\subseteq (x-h,x+h)$, in particular we have proved the last assertion of the claim about the directions.

\step{III}{The bi-Lipschitz property and the $L^\infty$ estimate for $\varphi_{st}$.}
To conclude the proof we only have to check that, whenever $x$ is a Lebesgue point for $\varphi'$ and the points $s$ and $t$ are in $I_\ell(x)$, the function $\varphi_{st}$ is $(L+\eps)$-biLipschitz and satisfies $\|\varphi-\varphi_{st}\|_{L^\infty} < \eps$.\par

The $L^\infty$ estimate comes directly by Step~I, keeping in mind~(\ref{stimaLinfinito}) and since by construction $2L|t-s| \leq 4\ell L < 4\eps h L < \eps$; moreover, Step~I ensures also the Lipschitz property, even with constant $L$ instead of $(L+\eps)$: as a consequence, we only have to take care of the inverse Lipschitz inequality. In other words, we take $y,\, z\in [0,1]$ and we have to check that
\begin{equation}\label{biLip}
|\varphi_{st}(y)-\varphi_{st}(z) | \geq \frac {|y-z|}{L+\eps} \,.
\end{equation}
If $y$ and $z$ are both in $[0,1]\setminus (s,t)$, then~(\ref{biLip}) is true --with $L$ in place of $L+\eps$-- because $\varphi=\varphi_{st}$ at both $y$ and $z$, while if they are both in $[s,t]$ then the validity of~(\ref{biLip}) --again with $L$ in place of $L+\eps$-- can be obtained exactly as in~(\ref{bothsides}). Without loss of generality, let us then consider the case when $s<y<t<z$, which we further subdivide in two cases.\par

If $z \in (x-h,x+h)$ then, as observed at the end of Step~II, the angle $\theta=\angle{\varphi(y)}{\varphi(t)}{\varphi(z)}$ is at most $2\eps$. Recalling that the validity of~(\ref{biLip}) with $L$ in place of $L+\eps$ is already known for both the pairs $(y,t)$ and $(t,z)$, we have then
\[\begin{split}
|\varphi_{st}(y)-\varphi_{st}(z)| &\geq \cos(\theta/2) \big(|\varphi_{st}(y)-\varphi_{st}(t)|+|\varphi_{st}(t)-\varphi_{st}(z)|\big)\\
&\geq \cos \eps \, \frac{|y-t|+|t-z|}L
\geq \frac{|y-z|}{L+\eps}\,,
\end{split}\]
which is valid up to take $\eps$ small enough.\par

Finally, assume that $z>x+h$: in this case it is enough to observe that, also by~(\ref{stimaLinfinito}),
\[\begin{split}
\frac{|\varphi_{st}(y)-\varphi_{st}(z)|}{|y-z|}&=
\frac{|\varphi_{st}(y)-\varphi(z)|}{|y-z|}
\geq \frac{|\varphi(y)-\varphi(z)|}{|y-z|} - \frac{\|\varphi_{st}-\varphi\|_{L^\infty}}{|y-z|}\\
&\geq \frac 1L - \frac{4 L \ell}{h-\ell}
= \frac 1L - \frac{4 L \tilde\eps}{1-\tilde\eps}
> \frac 1{L+\eps}\,,
\end{split}\]
up to have chosen $\tilde\eps=\tilde\eps(\eps, L)$ small enough. Thus, the estimate~(\ref{biLip}) has been proved in any case and the proof is concluded.
\end{proof}

\begin{definition}
Given an interval $J=(a,b)\subseteq [0,1]$ and $\eps>0$, we call \emph{central part of $J$} the interval $J^\eps$ given by
\[
J^\eps := \bigg(\frac{a+b}2 - \frac \eps 2\,(b-a), \frac{a+b}2 + \frac \eps 2\,(b-a)\bigg)\,.
\]
Moreover, we say that $J$ is \emph{$\eps$-admissible} if there exists $x\in J^\eps$ such that $\lbar(x) > (b-a)/2$.
\end{definition}

\proofof{Proposition~\ref{step1}}
For any $N\in\N$, we write $[0,1]$ as the essentially disjoint union of the intervals $J_m = \big(m/N, (m+1)/N\big)$, with $0\leq m < N$. Moreover, we let $\tilde \eps=\tilde\eps(\eps,L)<\eps$ be a small constant, to be specified later. We split the proof in three steps for clarity.

\step{I}{A piecewise linear $(L+\tilde\eps)$-biLipschitz function $\varphi_m$ on each $\tilde\eps$-admissible interval.}
Let us start by considering an interval $J_m$ which is $\tilde\eps$-admissible. Then, there exists a Lebesgue point $x_m\in J_m^{\tilde\eps}$ for $\varphi'$ satisfying $\lbar(x_m)>1/2N$. Let now $\ell= {\rm dist} (x_m,[0,1]\setminus J_m)$: of course $\ell \leq 1/2N<\bar\ell(x_m)$, hence we can apply Lemma~\ref{Lebesgue} with $\tilde\eps$ in place of $\eps$, and get two points
\begin{align*}
x_m^- \in \bigg(\frac m N, \frac mN+\frac{2\tilde\eps}N\bigg)\,, && 
x_m^+ \in \bigg(\frac {m+1}N-\frac{2\tilde\eps}N,\frac {m+1}N \bigg)
\end{align*}
such that the function $\varphi_m=\varphi_{x_m^-x_m^+}$ of Definition~\ref{varphist} is $(L+\tilde\eps)$-biLipschitz and satisfies $\|\varphi-\varphi_m\|_{L^\infty}< \tilde\eps<\eps$. Notice that $\varphi_m$ is piecewise linear on a subset of $J_m$ having length at least $(1-4\tilde\eps)/N$. We underline now another $L^\infty$ estimate which holds for $\varphi-\varphi_m$, which will be needed later; namely, since $\varphi$ and $\varphi_m$ are bi-Lipschitz and they coincide at $x_m^-$, then for any $y\in (x_m^-,x_m^+)$ we have
\begin{equation}\label{betterest}
|\varphi_m(y)-\varphi(y)| \leq |\varphi_m(y)-\varphi_m(x_m^-)| + |\varphi(x_m^-)-\varphi(y)|
\leq (2L+\tilde\eps)(y-x_m^-) < \frac{3L}N\,.
\end{equation}

\step{II}{The length of the non $\tilde\eps$-admissible intervals $J_m$ is small.}
Let us consider an interval $J_m$ which is not $\tilde\eps$-admissible. By definition, this means that no Lebesgue point $x$ in $J_m^{\tilde\eps}$ satisfies $\lbar(x)>1/2N$, or equivalently that $J_m^{\tilde\eps}$ is entirely contained in
\[
A_N=\bigg\{x\in [0,1]:\, \hbox{either $x$ is not a Lebegue point for $\varphi'$, or } \lbar(x)\leq \frac 1{2N}\bigg\}\,.
\]
As a consequence, the union of the intervals which are not $\tilde\eps$-admissible has length at most $|A_N|/\tilde\eps$: hence, since $\eps>0$ is fixed and since $\tilde\eps$ will ultimately depend only on $\eps$ and $L$, by Rademacher Theorem we can select $N\gg 1$ such that this union is as small as we wish.

\step{III}{Definition of the function $\varphi_{\eps}$.}
We are now in position to define the desired function $\varphi_\eps$. More precisely, we let $\varphi_\eps=\varphi_m$ in every $\tilde\eps$-admissible interval $J_m$, and $\varphi_\eps=\varphi$ on the other intervals; thus, $\varphi_\eps$ coincides with $\varphi$ on every interval which is not $\tilde\eps$-admissible, as well as in the ``external'' portion $J_m\setminus (x_m^-,x_m^+)$ of the $\tilde\eps$-admissible intervals.\par

First of all, observe that $\varphi_\eps$ is piecewise linear on the union of the intervals $(x_m^-,x_m^+)$, hence --by Steps~I and~II-- on a portion of $[0,1]$ having measure larger than $1-5\tilde\eps$ (thus in turn larger than $1-\eps$ if $\tilde\eps<\eps/5$) as soon as $N\gg 1$.\par

Second, by construction we have $\varphi_\eps(0)=\varphi(0)$ and $\varphi_\eps(1)=\varphi(1)$; moreover, since in every $\tilde\eps$-admissible interval $J_m$ one has $\|\varphi-\varphi_\eps\|_{L^\infty}=\|\varphi-\varphi_m\|_{L^\infty}<\tilde\eps$, while on each non $\tilde\eps$-admissible interval one has $\varphi_\eps=\varphi$, the $L^\infty$ estimate and thus the whole (\ref{claimmain}) has been established.\par

To conclude, we have only to check the $(L+\eps)$-biLipschitz property of $\varphi_\eps$. To do so, having fixed two points $y<z$ in $[0,1]$, we need to show that
\begin{align}\label{doubleineq}
|\varphi_\eps(y) -\varphi_\eps(z)| \leq (L+\eps) |y-z|\,, && |\varphi_\eps(y) -\varphi_\eps(z)| \geq \frac 1{L+\eps}\, |y-z|\,.
\end{align}
Since $\varphi$ is $L$-biLipschitz by assumption, and every $\varphi_m$ is $(L+\tilde\eps)$-biLipschitz by Step~I, there is nothing to prove unless $y \in (x_m^-,\, x_m^+)$ and $z\in (x_n^-,x_n^+)$ for some $m<n$, being both the intervals $J_m$ and $J_n$ $\tilde\eps$-admissible. In this case, the first inequality in~(\ref{doubleineq}) comes directly by the triangular inequality, being
\[\begin{split}
|\varphi_{\eps}(y) - \varphi_{\eps}(z)| &\leq |\varphi_{\eps}(y) - \varphi_{\eps}(x_{m}^{+})| + |\varphi_{\eps}(x_{m}^{+}) - \varphi_{\eps}(x_{n}^{-})| + |\varphi_{\eps}(x_{n}^{-}) - \varphi_{\eps}(z)| \\
&= |\varphi_m(y) - \varphi_m(x_{m}^{+})| + |\varphi(x_{m}^{+}) - \varphi(x_{n}^{-})| + |\varphi_n(x_{n}^{-}) - \varphi_n(z)|\\
&\leq (L+\tilde\eps) | y- z |
\leq (L+\eps) | y- z |\,.
\end{split}\]
To show the other inequality, it is convenient to distinguish two subcases, namely, whether $y$ and $z$ are very close, or not. More precisely, let us first assume that $z<x_m + 1/(2N\tilde\eps)<x_m+\bar\ell(x_m)/\tilde\eps$; in this case, by Lemma~\ref{Lebesgue} we know that the angle $\theta=\angle{\varphi_{\eps}(y)}{\varphi_{\eps}(x_{m}^{+})}{\varphi_{\eps}(z)}$ satisfies $\theta> \pi-2\tilde\eps$, so that as soon as $\tilde\eps=\tilde\eps(\eps,L)$ is small enough we have
\[
|\varphi_{\eps}(y) - \varphi_{\eps}(z)| \geq \cos(\tilde\eps)\big( |\varphi_{\eps}(y) - \varphi_{\eps}(x_{m}^{+})|  + |\varphi_{\eps}(x_{m}^{+}) - \varphi_{\eps}(z)|\big)
\geq \frac{\cos(\tilde\eps)}{L+\tilde\eps}\, | y- z |
\geq \frac 1{L+\eps}\, | y- z |\,.
\]
Finally, if $z\geq x_m+1/(2N\tilde\eps)$, then of course $|y-z|\geq 1/(3N\tilde\eps)$. As a consequence, since by~(\ref{betterest}) we have $\|\varphi-\varphi_\eps\|_{L^\infty}<3L/N$, we get
\[
\frac{|\varphi_{\eps}(y) - \varphi_{\eps}(z)|}{| y-z |} 
\geq \frac{|\varphi(y) - \varphi(z)|}{| y-z |} - \frac{2 \|\varphi-\varphi_\eps\|_{L^\infty}}{| y-z |}
\geq \frac 1L - 18L\tilde\eps \geq \frac 1{L+\eps}\,,
\]
where the last inequality is again true for a suitable choice of $\tilde\eps=\tilde\eps(\eps,L)$. The second inequality in~(\ref{doubleineq}) is thus proved in any case, and the proof is concluded.
\end{proof}

\begin{remark}\label{C1}
Notice that, if the function $\varphi$ is ${\rm C}^1$ up to the boundary on the interval $[0,1]$, then Lemma~\ref{Lebesgue} can be applied to any point of $[0,1]$, thus by a trivial compactness argument the proof of Proposition~\ref{step1} can be modified to get an $(L+\eps)$-biLipschitz approximation of $\varphi$ which is finitely piecewise linear on the whole $[0,1]$.
\end{remark}

\section{The ``non Lebesgue intervals''\label{sect3}}

In this section we show that any $L$-biLipschitz function $\varphi$ can be modified inside any small interval $(a,b)$, shrinking a little bit this interval, becoming ${\rm C}^1$ there, and remaining globally $L$-biLipschitz. In the next section we will apply this result to the ``non Lebesgue intervals'', that is, the intervals which we were not able to treat in the last section. The main aim of the section is to prove the following result.

\begin{prop}\label{nLI}
Let $\varphi:[0,C]\to \R^2$ be an $L$-biLipschitz function, let $[a,b]\subseteq [0,C]$ be a given interval, and suppose that for some $\eps>0$ the function $\varphi$ is linear on $(a-\eps,a)\cap [0,C]$ and on $(b,b+\eps)\cap [0,C]$, with $|\varphi'|=L$ on both these intervals. Then, there exists $a+(b-a)/L^2\leq b'\leq b$ and an $L$-biLipschitz function $\psi:[0,C-(b-b')]\to \R^2$ which is ${\rm C}^1$ on $[a,b']$ and satisfies
\begin{equation}\label{propprel}\left\{
\begin{array}{ll}
\psi(t) = \varphi(t) &\hbox{for every $0\leq t\leq a$}\,,\\
\psi(t) = \varphi(t+b-b') \qquad &\hbox{for every $b'\leq t\leq C-(b-b')$}\,.
\end{array}
\right.\end{equation}
\end{prop}

To obtain this result, the following two definitions will be useful.
\begin{definition}[Fast and short functions]\label{deffs}
Let $\varphi:[0,C]\to\R^2$ be an $L$-biLipschitz function and $[a,b]\subseteq [0,C]$ be a given interval. We say that a function $\psi:[0,C-(b-b')]\to \R^2$ is \emph{fast on $[a,b]$} if $a+(b-a)/L^2\leq b'\leq b$, $\psi$ satisfies~(\ref{propprel}),
\begin{align}\label{mildbiLip}
\frac 1L\, |z-y| \leq  |\psi(z)-\psi(y)| \leq L |z-y|
&& \forall\, y \in [a,b']\,, z \notin  [a,b'] \,,
\end{align}
and $|\psi'|\equiv L$ on $[a,b']$. Moreover, any $\psi$ which minimizes the value of $b'$ among all the functions fast on $[a,b]$, is said to be \emph{short on $[a,b]$}.
\end{definition}

In words, a ``fast'' function is a function which connects $\varphi(a)$ with $\varphi(b)$ always moving at maximal speed, and satisfying~(\ref{mildbiLip}),  while a ``short'' function is the shortest possible fast function. Let us immediately make a very simple observation, which we will use often later.
\begin{lemma}\label{ifnotshort}
Let $\varphi:[0,C]\to\R^2$ be an $L$-biLipschitz function, and let $\psi:[0,C-(b-b')]\to\R^2$ be short on some interval $[a,b]\subseteq [0,C]$. Let also $a\leq r < s \leq b'$, and assume that $\psi$ is not a straight line between $\psi(r)$ and $\psi(s)$. Then, the inverse $L$-Lipschitz property for the function $\psi_{rs}^+$ fails for some $p\notin [a,b'-(s-s^+)]$ and $q\in (r,s^+)$, where $\psi_{rs}^+$ and $s^+$ are as in Definition~\ref{varphist}.
\end{lemma}
\begin{proof}
Let us consider the function $\psi_{rs}^+:[0,C-(b-b'')]\to \R^2$, with $b''=b'-(s-s^+)$, which of course satisfies~(\ref{propprel}). Since $\psi$ is not a straight line between $\psi(r)$ and $\psi(s)$, we have that $b''<b'$ and then, since $\psi$ is short on $(a,b)$, by definition we get that $\psi_{rs}^+$ cannot be fast on $(a,b)$. As a consequence, recalling~(\ref{mildbiLip}), we know that there must be some $p\notin [a,b'']$ and some $q\in [a,b'']$ such that the $L$-biLipschitz property for $\psi_{rs}^+$ fails at $p$ and $q$. However, we know that $\big|(\psi_{rs}^+)'\big|=L$ in $(a,b'')$, while outside $\big|(\psi_{rs}^+)'\big|\leq L$ since $\psi_{rs}^+$ coincides with $\varphi$ up to a translation of the variable, and $\varphi$ is $L$-biLipschitz. Thus, the $L$-Lipschitz property for $\psi_{rs}^+$ cannot fail, and we realize that the inverse $L$-Lipschitz property must fail at $p$ and $q$. By symmetry, we can also assume that $p<a$; hence, if $a\leq q \leq r$, then
\[
\frac{|\psi_{rs}^+(q)-\psi_{rs}^+(p)|}{q-p} = 
\frac{|\psi(q)-\psi(p)|}{q-p} \geq \frac 1L\,,
\]
because the function $\psi$ is short and then in particular it satisfies~(\ref{mildbiLip}). Instead, if $s^+\leq q \leq b''$, then we have
\[\begin{split}
\frac{|\psi_{rs}^+(q)-\psi_{rs}^+(p)|}{q-p} &= 
\frac{|\psi(q+(s-s^+))-\psi(p)|}{q-p}
=\frac{|\psi(q+(s-s^+))-\psi(p)|}{q+(s-s^+)-p}\,\frac{q-p+(s-s^+)}{q-p}\\
&\geq\frac{|\psi(q+(s-s^+))-\psi(p)|}{q+(s-s^+)-p}
\geq \frac 1L\,.
\end{split}\]
As a consequence, we obtain that $q$ must be in $(r,s^+)$, and the thesis is concluded.
\end{proof}

Our next result tells that a short function always exists, and it is even $L$-biLipschitz: notice that this is not guaranteed by~(\ref{mildbiLip}), since there we check only some pairs $(y,z)$, namely, those for which $y$ is inside the interval $(a,b')$ and $z$ is outside it.

\begin{lemma}\label{short->biLip}
Let $\varphi:[0,C]\to \R^2$ be an $L$-biLipschitz function, and let $[a,b]\subseteq [0,C]$ be a given interval. Then, there exists a function $\psi$ short on $[a,b]$, and any such function is $L$-biLipschitz.
\end{lemma}
\begin{proof}
First of all, let us observe that the set of the fast functions is not empty. Indeed, the function $\varphi$ itself, reparametrized at speed $L$ in $(a,b)$, is fast: more precisely, let us set
\[
b' = a+  \frac{\H^1\big(\arc{\varphi(a)\varphi(b)}\big)}L\,,
\]
let $\sigma:[0,C]\to [0,C-(b-b')]$ be the one-to-one function given by
\[
\sigma(t) = \left\{
\begin{array}{ll}
t &\forall\, 0\leq t\leq a\,,\\
\bal a+\frac{\H^1\big(\arc{\varphi(a)\varphi(t)}\big)}L \eal \quad &\forall\,a<t<b\,,\\
t-(b-b') & \forall\, b\leq t\leq C\,,
\end{array}\right.
\]
and set $\psi_1$ as $\psi_1(\sigma(t))=\varphi(t)$. We claim that $\psi_1$ is a fast function on $[a,b]$: everything is obvious by construction except the validity of~(\ref{mildbiLip}). But in fact, let $y\in (a,b')$ and $z>b'$ (if $z<a$, the very same argument applies). Since $|\psi_1'(t)|=L$ for $t\in (y,b')$, while for $b'<t<z$ one has $|\psi_1'(t)|=|\varphi'(t+b-b')|\leq L$ because $\varphi$ is $L$-biLipschitz, we get immediately the validity of the second inequality. Concerning the first one, we have just to recall that $|\sigma'|\leq 1$, so that
\[
b'-y =\sigma(b) - \sigma\big(\sigma^{-1}(y)\big)\leq b - \sigma^{-1}(y)\,,
\]
and then we directly get
\[
\frac{|\psi_1(z)-\psi_1(y)|}{|z-y|}=
\frac{|\varphi(z-b'+b)-\varphi(\sigma^{-1}(y))|}{z-b'+b'-y}
\geq \frac{|\varphi(z-b'+b)-\varphi(\sigma^{-1}(y))|}{z-b'+b-\sigma^{-1}(y)}\geq \frac 1L\,
\]
where in the last inequality we have used the bi-Lipschitz property of $\varphi$. So, also the first inequality in~(\ref{mildbiLip}) is proved and thus the claim is established.\par
To get the existence of a short function, it is enough to recall that all the fast functions are uniformly continuous on uniformly bounded intervals; thus, such existence follows directly by Ascoli--Arzel\`a Theorem and since any uniform limit of fast functions is also fast.\par
To conclude, we take a short function $\psi$ on $[a,b]$, and we have to show that $\psi$ is $L$-biLipschitz. We have already noticed that $|\psi'|\leq L$, so the Lipschitz property is obvious and we only have to care about the inverse Lipschitz property. To do so, let us take $y<z \in [0,C-(b-b')]$. If $y$ and $z$ are both smaller than $a$, or both larger than $b'$, this comes directly by the inverse Lipschitz property of $\varphi$; if one between $y$ and $z$ is smaller than $a$, and the other is larger than $b'$, the same argument applies since
\[
|\psi(z)-\psi(y)| = \big|\varphi(z+b-b') - \varphi(y)\big| \geq \frac 1L\, \big|(z+b-b')-y\big| \geq \frac 1L\, |z-y|\,;
\]
if exactly one between $y$ and $z$ is in $[a,b']$, the inequality is ensured by~(\ref{mildbiLip}). Summarizing, the only situation left to prove is the case $a<y<z<b'$.\par
Let us assume, by contradiction, that there exists $a<r<s<b'$ such that
\begin{equation}\label{contr1}
|\psi(s)-\psi(r)| < \frac 1L\, |s-r|\,,
\end{equation}
and consider the function $\psi_{rs}^+:[0,C-(b-b'')]\to \R^2$ given by Definition~\ref{varphist}. Notice that $\psi$ cannot be a straight line between $\psi(r)$ and $\psi(s)$, by~(\ref{contr1}) and the fact that $|\psi'|=L$ on $(a,b')$. Thus, Lemma~\ref{ifnotshort} ensures the existence of some $p\notin [a,b'']$ (and, without loss of generality, we can think $p<a$) and $q\in (r,s^+)$ such that
\begin{equation}\label{contr2}
|\psi_{rs}^+(q)-\psi_{rs}^+(p)| < \frac 1L\, (q-p)\,.
\end{equation}
Finally, making use of the validity of~(\ref{mildbiLip}) for $\psi$, together with~(\ref{contr1}) and~(\ref{contr2}), we get
\[\begin{split}
\frac 1L\, |s-p| &\leq  |\psi(s)-\psi(p)|
\leq  |\psi(s)-\psi_{rs}^+(q)| + |\psi_{rs}^+(q)-\psi(p)|\\
&= |\psi(s)-\psi(r)| - |\psi_{rs}^+(q)-\psi_{rs}^+(r)| + |\psi_{rs}^+(q)-\psi_{rs}^+(p)|
<\frac 1L\, (s-r) - L(q-r) + \frac 1L\, (q-p)\\
&\leq \frac 1L |s-p|\,,
\end{split}\]
and since the contradiction shows that $\psi$ is $L$-biLipschitz, the proof is concluded.
\end{proof}

Keep in mind that we aim to prove Proposition~\ref{nLI}, that is, we want to find some $L$-biLipschitz function $\psi$ which satisfies~(\ref{propprel}) and which is ${\rm C}^1$ on $[a,b']$. By definition, any function fast in $[a,b]$ already satisfies~(\ref{propprel}), and the lemma above ensures that any function short on $[a,b]$ is also $L$-biLipschitz. We will then get our proof once we show that any short function is necessarily ${\rm C}^1$ on $[a,b']$. To do so, we start with a couple of preliminary geometric estimates.
\begin{lemma}\label{3.4}
For every small constants $\ell$ and $\eta$ and for every $L\geq 1$, there exists $\bar\delta(\ell,\eta,L)\ll \ell$ satisfying the following properties. Let $P,\,Q,\,S$ be three points in $\R^2$ such that $\segm{PQ}\geq\ell/2$ and $\delta=\segm{QS}\leq\bar\delta$. Call also $\theta,\,\theta',\,\nu\in\S^1$ the directions of the segments $PQ$, $PS$ and $QS$ respectively. Then the following holds true:
\begin{gather}
|\theta-\theta'| \leq \frac \eta{L^2}\,,\label{prop1}\\
\theta\cdot \nu - \frac\eta{L^2} \leq \frac{\segm{PS}-\segm{PQ}}\delta \leq \theta\cdot \nu + \frac\eta{L^2}\,.\label{prop2}
\end{gather}
\end{lemma}
\begin{proof}
Once $\ell$, $\eta$  and $L$ are given, the existence of some $\bar\delta$ satisfying~(\ref{prop1}) is immediate by continuity; we will show that the same choice of $\bar\delta$ gives also~(\ref{prop2}).\par
Let us call $\tau:[0,\delta]\to \R^2$ the function given by $\tau(t)=Q + t\nu$, so that $\tau(0)=Q$ and $\tau(\delta)=S$; call also $\theta(t)$ the direction of the segment $P\tau(t)$, and observe that $|\theta(t)-\theta|\leq \eta/L^2$ by~(\ref{prop1}) applied to the triple $(P,Q,\tau(t))$. Hence,
\[
\segm{PS}-\segm{PQ}=\int_0^\delta \frac{d}{dt}\, \Big( \segm{P\tau(t)}\Big) \, dt
=\int_0^\delta \theta(t)\cdot \nu\,dt
= \delta \theta\cdot \nu + \int_0^\delta (\theta(t)-\theta)\cdot \nu\,dt\,,
\]
and the modulus of the latter integral is smaller than $\delta\eta/L^2$, hence~(\ref{prop2}) follows.
\end{proof}

\begin{lemma}\label{smallone}
Let $\ell$, $\eta$  and $L$ be fixed, let $\varphi:[0,C]\to\R^2$ be an $L$-biLipschitz function, and take three points $P=\varphi(p)$, $Q=\varphi(q)$, $R=\varphi(r)$, with $p<q<r$, satisfying $\segm{PQ}\geq \ell$ and $\delta:=\segm{QR}\leq \bar\delta(\ell,\eta,L)$. Assume that the function $\varphi^+_{qr}:[0,C-(r-r^+)]\to \R^2$ of Definition~\ref{varphist} does not satisfy the inverse $L$-Lipschitz property at the pair $(p,t)$ with some $q\leq t\leq r^+$. Then,
\begin{gather}
\frac 1{L^2}-2\,\frac\eta{L^2} \leq \theta\cdot \nu \leq \frac 1{L^2}+ 2\,\frac\eta{L^2}\,,\label{prop3}\\
\frac{\segm{PQ}}{|q-p|} \leq \frac 1 L + \frac{3\eta\delta}{|q-p| L^2}\,,\label{prop4}\\
\frac{\H^1(\arc{QR})}{\segm{QR}} \leq 1+6\eta\,,\label{prop5}
\end{gather}
where $\theta$ and $\nu$ are the directions of the segments $PQ$ and $QR$ respectively.
\end{lemma}
\begin{proof}
First of all, let us call for brevity $\sigma:=\segm{Q\varphi^+_{qr}(t)}$ and $Q_\sigma=\varphi^+_{qr}(t)$. The failure of the inverse $L$-Lipschitz property at $p$ and $t$, together with~(\ref{prop2}) applied with $S=Q_\sigma$ and with the fact that $\varphi$, instead, is $L$-biLipschitz, gives
\begin{equation}\label{ty}
\frac\sigma{L^2}+\frac{q-p}L=\frac{t-p}L>\segm{\varphi^+_{qr}(p)\varphi^+_{qr}(t)}=\segm{PQ_\sigma}
\geq\segm{PQ}+\sigma\bigg(\theta\cdot \nu - \frac\eta{L^2}\bigg)
\geq\frac{q-p}L+\sigma\bigg(\theta\cdot \nu - \frac\eta{L^2}\bigg)\,,
\end{equation}
which can be rewritten as
\[
\theta\cdot \nu < \frac 1{L^2} + \frac\eta{L^2}\,,
\]
so that one inequality in~(\ref{prop3}) is already proved.\par
Notice now that $\segm{PQ_\sigma}\geq \ell-\bar\delta>\ell/2$, so we can apply~(\ref{prop2}) also to $Q=Q_\sigma$ and $S=R$. By~(\ref{ty}) and the biLipschitz property of $\varphi$ again, we have then
\[\begin{split}
\frac {\sigma-L(r^+-q)}{L^2}+\frac{r-p}L
&\geq\frac {\sigma-L(r-q)}{L^2}+\frac{r-p}L
=\frac \sigma{L^2}+\frac{q-p}L>\segm{PQ_\sigma}\\
&\geq\segm{PR}-(\segm{QR}-\sigma)\bigg(\tilde\theta\cdot \nu + \frac\eta{L^2}\bigg)\\
&\geq \frac{r-p}L+(\sigma-L(r^+-q))\bigg(\tilde\theta\cdot \nu + \frac\eta{L^2}\bigg)\,,
\end{split}\]
where $\tilde\theta$ is the direction of $PQ_\sigma$. Since $L(r^+-q)=\segm{QR}\geq \sigma$, we deduce
\[
\tilde\theta\cdot\nu > \frac 1{L^2} -\frac\eta{L^2}\,.
\]
And since $|\tilde\theta-\theta|<\eta/L^2$ by~(\ref{prop1}), we conclude the validity of~(\ref{prop3}).\par

Property~(\ref{prop4}) can be directly deduced from~(\ref{ty}) and~(\ref{prop3}), since
\[
\frac{\segm{PQ}}{|q-p|} < \frac 1 L +\frac\sigma{|q-p|}\,\bigg(\frac 1{L^2} - \theta\cdot \nu + \frac\eta{L^2}\bigg)
\leq \frac 1 L +\frac{3\eta\sigma}{|q-p|L^2}
\leq \frac 1 L +\frac{3\eta\delta}{|q-p|L^2}\,.
\]
And finally, to get property~(\ref{prop5}), first we use that $\varphi$ is $L$-biLipschitz to get
\begin{equation}\label{quasopra}
\segm{PR}\geq \frac{r-p}{L} = \frac{q-p}L + \frac{r-q}L \geq \frac{q-p}L + \frac{\H^1(\arc{QR})}{L^2}\,,
\end{equation}
and then we use~(\ref{prop2}), (\ref{prop4}) and~(\ref{prop3}) to get
\[
\segm{PR} \leq \segm{PQ} + \delta \bigg( \theta\cdot \nu + \frac \eta{L^2}\bigg)
\leq \frac {q-p} L + \frac\delta{L^2} ( 1+ 6\eta)\,,
\]
which inserted in~(\ref{quasopra}) gives~(\ref{prop5}).
\end{proof}

Let us now present the main technical tool of this section: thanks to this result, we will be able to prove the regularity of any short map.

\begin{lemma}\label{bigone}
Let $\ell$, $\eta$  and $L$ be fixed, let $\varphi:[0,C]\to\R^2$ be an $L$-biLipschitz function with $|\varphi'|\equiv L$ in $(q,r)$, and take five points $P=\varphi(p)$, $Q=\varphi(q)$, $R=\varphi(r)$, $Q'=\varphi(q')$ and $Q''=\varphi(q'')$ with $p<q<q'<q''<r$, satisfying $\segm{PQ}\geq 2\ell$ and $\delta:=\segm{QR}\leq \bar\delta(\ell,\eta,L)$. Assume also that both $Q'$ and $Q''$ have distance at least $\delta/3$ from each of $Q$ and $R$, that $\eta$ is small with respect to $L-1$ and $1/L$, and that~(\ref{prop5}) holds true. Then, if the inverse $L$-Lipschitz property for the function $\varphi_{q'q''}^+$ is not satisfied by $P$ and some point in the segment $Q'Q''$, one has
\begin{equation}\label{oscest}
|\nu-\nu'| \leq \frac 12\, \min \bigg\{\frac 1{L^2}, 1-\frac 1{L^2}\bigg\}\qquad \Longrightarrow \qquad |\nu-\nu'| \leq 15\sqrt\eta\,,
\end{equation}
where $\nu$ and $\nu'$ are the directions of the segments $QR$ and $Q'Q''$ respectively.
\end{lemma}
\begin{proof}
First of all, we use property~(\ref{prop5}) for $Q$ and $R$, which is valid by assumption: we immediately get that every point in the curve $\arc{QR}$, hence in particular both $Q'$ and $Q''$, has distance less than $2\delta\sqrt\eta$ from the segment $QR$. Since $\segm{QQ'}\geq \delta/3$ and $\segm{Q''R}\geq \delta/3$, we deduce that
\begin{align}\label{almuno}
|\nu-\tilde\nu| \leq 6\sqrt\eta\,, && \segm{Q'Q''}\leq \frac \delta 2\leq \frac 32\, \segm{QQ'}\,,
\end{align}
being $\tilde \nu$ the direction of $QQ'$. Moreover, the validity of~(\ref{prop5}) also implies that
\[
\segm{QR}(1+6\eta) \geq \H^1(\arc{QR}) 
= \H^1(\arc{QQ'})+\H^1(\arc{Q'R})
\geq \H^1(\arc{QQ'})+\segm{Q'R}
\geq \H^1(\arc{QQ'})+\segm{QR}-\segm{QQ'}\,,
\]
which in turn implies by the assumptions
\begin{equation}\label{1207}
\H^1(\arc{QQ'}) \leq 6\eta \segm{QR} + \segm{QQ'} \leq \segm{QQ'}(1+18\eta)\,.
\end{equation}
Let us now use the fact that the inverse $L$-Lipschitz property for the function $\varphi_{q'q''}^+$ fails for $P$ and some point in $Q'Q''$. As a consequence, we can apply Lemma~\ref{smallone} to the points $P,\, Q'$ and $Q''$, so by~(\ref{prop3}), calling $\theta'$ the direction of $PQ'$, we know
\begin{equation}\label{alkn}
\Big|\theta' \cdot \nu' - \frac 1{L^2}\Big| \leq 2\,\frac\eta{L^2}\,.
\end{equation}
Let us now assume that
\begin{equation}\label{soquesto}
|\nu-\nu'| \leq \frac 12\, \min\bigg\{\frac 1{L^2}, 1-\frac 1{L^2}\bigg\}\,,
\end{equation}
so that the proof will be concluded once we show that
\begin{equation}\label{bastaquesto}
|\nu-\nu'| \leq 15\sqrt\eta\,.
\end{equation}
First of all, putting together~(\ref{alkn}), (\ref{soquesto}) and~(\ref{almuno}), a simple geometric argument shows that the directions $\nu$ and $\tilde\nu$ are in the same quadrant with respect to $\theta'$ as soon as $\eta$ is small enough; more precisely, this holds true as soon as $\eta$ is much smaller than both $L-1$ and $1/L$ (keep in mind that $L>1$). As a consequence, a trigonometric argument immediately gives
\[
\big|\theta' \cdot \tilde\nu - \theta' \cdot \nu'\big| \geq \frac{|\tilde\nu-\nu'|^2}3\,.
\]
Just to fix the ideas, we can assume that
\begin{equation}\label{siqu}
\theta' \cdot \tilde\nu - \theta' \cdot \nu' \geq \frac{|\tilde\nu-\nu'|^2}3\,,
\end{equation}
otherwise the argument below about $\segm{QQ'}$ has to be replaced by a completely similar argument about $\segm{RQ''}$.\par
Let us now collect all the information that we have: by (\ref{prop2}) applied to $P,\,Q'$ and $Q$, by~(\ref{prop4}) applied to $P,\,Q'$ and $Q''$, and by~(\ref{almuno}), we get
\[\begin{split}
\frac{q-p}L &\leq \segm{PQ} \leq \segm{PQ'}+\segm{QQ'} \bigg( \theta' \cdot (-\tilde\nu)+\frac \eta{L^2}\bigg)\\
&\leq \frac{q'-p}L + \frac{3 \eta \segm{Q'Q''}}{L^2} + \segm{QQ'} \bigg(\frac \eta{L^2} -\theta' \cdot \tilde\nu\bigg)
\leq \frac{q'-p}L + \frac{9 \eta \segm{QQ'}}{2L^2} + \segm{QQ'} \bigg(\frac \eta{L^2} -\theta' \cdot \tilde\nu\bigg)\\
&\leq\frac{q'-p}L + \segm{QQ'} \bigg(\frac{6\eta}{L^2} -\theta' \cdot \tilde\nu\bigg)\,.
\end{split}\]
This, also keeping in mind the fact that $|\varphi'|\equiv L$ in $(q,r)$ and~(\ref{1207}), implies
\[
\segm{QQ'} \bigg(\theta' \cdot \tilde\nu -\frac{6\eta}{L^2} \bigg)\leq \frac{q'-q}L 
=\frac{\H^1(\arc{QQ'})}{L^2}
\leq \frac{\segm{QQ'}(1+18\eta)}{L^2}\,,
\]
which finally gives
\[
\theta' \cdot \tilde\nu \leq \frac 1{L^2} + \frac{24\eta}{L^2}\,.
\]
Inserting now this estimate and~(\ref{alkn}) in~(\ref{siqu}), we get
\[
|\tilde\nu-\nu'| \leq \frac{9\sqrt\eta}L < 9\sqrt\eta\,,
\]
which together with~(\ref{almuno}) finally gives~(\ref{bastaquesto}).
\end{proof}

Even though the estimate~(\ref{oscest}) is quite obscure, it gives an important information, namely, if the directions $\nu$ and $\nu'$ are not too far away from each other, then they must be very close. We can now see that this implies the regularity of a short map $\varphi$. First of all, we prove the internal regularity in the open segment $(a,b)$.

\begin{lemma}\label{shortonC1}
Let $\varphi$ be an $L$-biLipschitz function, short on $[a,b]$. Then, $\varphi$ is of class ${\rm C}^1$ in the open interval $(a,b)$.
\end{lemma}
\begin{proof}
Let us take $a<a'<b'<b$, and let us fix some $\ell\ll 1$ such that
\[
\segm{PQ}>2\ell \quad \forall\, P=\varphi(p),\, Q=\varphi(q), \, p\notin [a,b],\, q\in (a',b')\,.
\]
We aim to show that $\varphi$ is of class ${\rm C}^1$ in $(a',b')$, and this will of course give the thesis since $a'$ and $b'$ are arbitrary.\par

Let us then fix a point $S=\varphi(s)$ with $s\in(a',b')$, and let also $\eta\ll 1$ be given. Define $\bar\delta=\bar\delta(\ell,\eta,L)$ according to Lemma~\ref{3.4}: we claim that there exists some direction $\nu\in\S^1$ for which
\begin{equation}\label{tgt}
\bigg|\frac{\varphi'(t)}L - \nu \bigg| \leq 16\sqrt\eta \qquad \forall\, t\in (a',b'):\, |t-s|\leq \frac{\bar\delta}{12L}\,.
\end{equation}
Since $s$ and $\eta$ are arbitrary, of course this will immediately imply the required ${\rm C}^1$ regularity of $\varphi$ in $(a',b')$.\par
\begin{figure}[thbp]
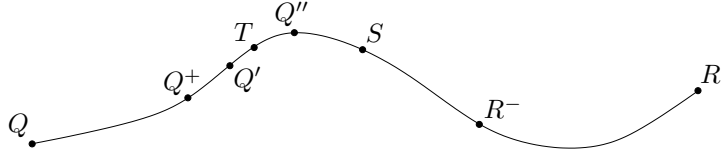
\caption{Situation in Lemma~\ref{shortonC1}}\label{figure}
\end{figure}
Let us now take $Q=\varphi(q)$ and $R=\varphi(r)$ in such a way that $\segm{QR}=\bar\delta$ and that $s=(q+r)/2$, and let us call $\nu$ the direction of the segment $QR$: Figure~\ref{figure} depicts all the involved points. Notice that, by construction and since $\bar\delta\ll \ell$, both $q$ and $r$ are in $(a,b)$, and both $\segm{PQ}$ and $\segm{PR}$ are larger than $\ell$ whenever $P=\varphi(p)$ for some $p\notin [a,b]$. If $\varphi$ is linear on $QR$, then of course $\varphi'$ is constantly equal to $L\nu$ in $(q,r)$, thus~(\ref{tgt}) is already established. Otherwise, Lemma~\ref{ifnotshort} says that there must be some $p\notin(a,b-(r-r^+))$ such that the inverse $L$-Lipschitz property for $\varphi_{qr}^+$ fails at $p$ and at some point in $(q,r^+)$. Thus, we can apply Lemma~\ref{smallone} and in particular~(\ref{prop5}) is true.\par
As noticed in the proof of Lemma~\ref{bigone}, this implies that the whole curve $\arc{QR}$ has distance less than $2\bar\delta\sqrt\eta$ from the segment $QR$. As a consequence, if we call $Q^+$ and $R^-$ the first and the last point of $\arc{QR}$ having distance $\bar\delta/3$ from $Q$ and from $R$ respectively, we clearly have
\begin{equation}\label{put1}
|\tilde\nu-\nu| \leq \arctan(12\sqrt\eta) < 16\sqrt\eta\,,
\end{equation}
where $\tilde\nu$ is the direction of the segment $Q^+R^-$. Let us assume now that~(\ref{tgt}) is false, thus there exists some $t$ with $|t-s|\leq \bar\delta/12L$ such that
\begin{equation}\label{put2}
\bigg|\frac{\varphi'(t)}L - \nu \bigg| > 16\sqrt\eta\,.
\end{equation}
Observe that, by construction, $t$ must belong to the interval $(q^+,r^-)$. By continuity, (\ref{put1}) and~(\ref{put2}) imply the existence of $q^+<q'<t<q''<r^-$ such that
\begin{equation}\label{put3}
|\nu' - \nu | =16\sqrt\eta\,,
\end{equation}
where $\nu'$ is the direction of the segment $Q'Q''$. The function $\varphi$ cannot be a segment between $Q'$ and $Q''$, because otherwise it would be $\varphi'(t) = L\nu'$, against~(\ref{put2}) and~(\ref{put3}). Again by Lemma~\ref{ifnotshort}, we deduce the existence of some new point $\widetilde P=\varphi(\tilde p)$ with $\tilde p \notin [a,b-(q''^+-q'')]$ such that the inverse $L$-Lipschitz property for $\varphi^+_{q'q''}$ fails at $\tilde p$ and some point between $q'$ and $q''^+$ (notice that there is no reason why this point $\widetilde P$ should coincide with the point $P=\varphi(p)$ of few lines ago).\par
We can then apply Lemma~\ref{bigone} with $P=\widetilde P$, and we get the validity of~(\ref{oscest}). Notice that, since $\eta$ has been taken arbitrarily small, by~(\ref{put3}) we can assume without loss of generality that
\[
|\nu-\nu'| = 16 \sqrt \eta \leq \frac 12\, \min \bigg\{\frac 1{L^2}, 1-\frac 1{L^2}\bigg\}\,.
\]
As a consequence, (\ref{oscest}) tells us that $|\nu-\nu'|\leq 15\sqrt\eta$, which is clearly a contradiction with~(\ref{put3}). Therefore, the proof of~(\ref{tgt}) is completed, and as noticed above the thesis follows.
\end{proof}

Now, we can extend the regularity up to the extremes of the interval $[a,b]$. We do this first for an interval compactly contained in $[0,C]$, and then for an interval reaching the boundary of $[0,C]$.

\begin{lemma}\label{rightleft}
Let $\varphi:[0,C]\to \R^2$ be an $L$-biLipschitz function, short on some $[a,b]\subset\subset [0,C]$, and assume that for some $\eps\ll 1$ the function $\varphi$ is linear on $(a-\eps,a)$ with $|\varphi'|\equiv L$. Then, $\varphi'$ is right-continuous in $a$.
\end{lemma}
\begin{proof}
Up to a rotation, we can assume for simplicity also that $\varphi$ is horizontal in $(a-\eps,a)$ or, in other words, that $\varphi' = L{\rm e}_1$ in $(a-\eps,a)$. Since by definition $|\varphi'|=L$ on the whole $(a,b)$, we have to find a direction $\bar\nu\in\S^1$ such that the directions of $\varphi'(t)$ converge to $\bar\nu$ when $t\searrow a$. First of all, let us define $\bar\theta = 2\arcsin (1/L^2)$, and notice that $\bar\theta$ converges to $\pi$ (resp., to $0$) if $L$ converges to $1$ (resp., to $+\infty$).
\step{I}{For every $r\in (a,a+\eps)$, the direction of the segment $AR$ is between $-(\pi-\bar\theta)$ and $\pi-\bar\theta$.}
Let us take a generic $r\in(a,a+\eps)$, define $\sigma=\segm{AR}/L$, and notice that
\begin{equation}\label{nsd}
L(r-a) = \H^1(\arc{AR})\geq \segm{AR}\,, \qquad \hbox{thus} \qquad r-a\geq \sigma\,.
\end{equation}
Call now $P=\varphi(p)=\varphi(a-\sigma)$, and notice that $\segm{PA}=\segm{AR}$ by construction, since the above inequality in particular ensures that $\sigma<\eps$. If we assume, by contradiction, that the direction of $AR$ is not between $-(\pi-\bar\theta)$ and $\pi-\bar\theta$, then in particular $\angle PAR<\bar\theta$. As a consequence, by~(\ref{nsd}) we have
\[
\segm{PR} < 2 \segm{AR}\sin (\bar\theta/2) = \frac{2\segm{AR}}{L^2} = \frac {2\sigma}L \leq \frac{r-p}L\,,
\]
and this gives a contradiction to the $L$-biLipschitz property of $\varphi$. Hence, the first step is concluded.

\step{II}{There exists $\bar\nu \in \big[-(\pi-\bar\theta),\pi-\bar\theta\big]$ such that the direction of $A\varphi(t)$ converges to $\bar\nu$ when $t\searrow a$.}
By compactness of $\S^1$, the directions of the segments $A\varphi(t)$ have at least a limit point for $t\searrow a$: the goal of this step is to show that there is actually only a single such limiting point.\par

Let us assume that the limiting directions are more than one: since the set of the limiting directions is clearly a connected subset of $\S^1$, which can only contain directions between $-(\pi -\bar\theta)$ and $\pi -\bar\theta$ by Step~I, we deduce in particular the existence of $\nu_1,\,\nu_2\in \S^1$, and of two sequences of points $R^i_n=\varphi(r^i_n)$ for $i\in \{1,\,2\}$ and $n\in\N$ satisfying
\begin{align*}
\nu_1,\, \nu_2 \in \big(-(\pi-\bar\theta),\pi-\bar\theta\big),\, &&
\nu_1 \neq \nu_2\,, &&
r^i_n \mathop{\searrow}\limits_{n\to \infty} a\,, &&
\frac{R^i_n-A}{|R^i_n-A|} = \nu_i\,.
\end{align*}
Let us now fix a constant $\ell$ much smaller than $L\eps$. For every $\eta>0$, calling for brevity $r=r^1_n$ and $R=R^1_n$, as soon as $n$ is big enough we have that $\segm{AR}$ is smaller than $\bar\delta(\ell,\eta,L)$. As a consequence, since of course $\varphi$ is not linear between $a$ and $r$ but it is short on $(a,b)$, Lemma~\ref{ifnotshort} implies that the inverse $L$-Lipschitz property for the function $\varphi_{ar}^+$ must fail for some pair $P_n=\varphi^+_{ar}(p_n)$ and $S=\varphi^+_{ar}(s)$, where $p_n\notin [a,b-(r-r^+)]$ while $s\in (a,r^+)$. Notice now the following simple general trigonometric fact: given two directions $\alpha_1,\,\alpha_2\in\S^1$, the map $\tau:\R\to\R^2$ defined as $\tau(x)=Lx\alpha_1$ for $x\geq 0$ and $\tau(x)=-Lx\alpha_2$ for $x<0$ is $L$-biLipschitz if and only if the angle between $\alpha_1$ and $\alpha_2$ is at least the angle $\bar\theta$ defined above. Thus, Step~I and the fact that $\nu_1\in \big(-(\pi-\bar\theta),\pi-\bar\theta\big)$, ensure that $p_n$ cannot belong to $(a-\eps,a)$ if $r<a+\eps$, which is of course true up to have taken $n$ big enough. As a consequence, we can apply Lemma~\ref{smallone} to the points $P=P_n,\, Q=A$ and $R$, and we get that
\begin{align*}
\bigg|\theta_n\cdot \nu_1 - \frac 1 {L^2} \bigg| \leq 2\,\frac\eta{L^2}\,, &&
\segm{AP_n} \leq \frac {|a-p_n|} L + \frac{3\eta\segm{AR}}{L^2}\,.
\end{align*}
where $\theta_n$ is the direction of the segment $P_nA$. If we now send $\eta\to 0$ and consequently $n\to \infty$, we find a point $P^1=\varphi(p^1)$ (which is a limit of some subsequence of the points $P_n$) such that, calling $\theta^1$ the direction of $P^1A$, it is
\begin{align}\label{titaglio}
\theta^1\cdot \nu_1 =\frac 1 {L^2}\,, &&
\segm{AP^1} =\frac {|a-p^1|} L\,.
\end{align}
The very same argument, using the direction $\nu_2$ in place of $\nu_1$, gives us of course another point $P^2=\varphi(p^2)$ satisfying
\begin{align}\label{ledita}
\theta^2\cdot \nu_2 =\frac 1 {L^2}\,, &&
\segm{AP^2} =\frac {|a-p^2|} L\,,
\end{align}
where $\theta^2$ is the direction of $P^2A$. Again recalling that the set of the limiting directions, among which we have chosen $\nu_1$ and $\nu_2$, is a connected subset of $\S^1$, from~(\ref{titaglio}), (\ref{ledita}) and the fact that $\nu_1\neq \nu_2$ we get that, up to swap $\nu_1$ and $\nu_2$,
\begin{equation}\label{poverine}
\theta^1\cdot \nu_2 < \frac {1-\eta}{L^2}
\end{equation}
for some strictly positive constant $\eta$. Let us now again select a point $R=R^2_n$ with $n$ big enough so that $\segm{AR}\leq \bar\delta(\ell,\eta,L)$, and let us assume that $p^1<a$ (otherwise one just has to repeat the argument below swapping the roles of $A$ and $R$). Recalling that $\varphi$ is $L$-biLipschitz, so in particular the inverse $L$-Lipschitz property holds for the points $p^1$ and $r=r^2_n$, using again Lemma~\ref{3.4} with $P=P^1$, $Q=A$ and $S=R$, and by~(\ref{titaglio}) and~(\ref{poverine}), we have
\[\begin{split}
\frac{\segm{AR}}{L^2} &\leq \frac{r-a}L = \frac{r-p^1}L-\frac{a-p^1}L 
\leq \segm{P^1R}-\frac{a-p^1}L 
\leq \segm{P^1A} + \segm{AR}\bigg(\theta^1\cdot \nu_2 + \frac\eta{L^2}\bigg)-\frac{a-p^1}L \\
&= \segm{AR}\bigg(\theta^1\cdot \nu_2 + \frac\eta{L^2}\bigg)
<\frac{\segm{AR}}{L^2}\,,
\end{split}\]
and the contradiction shows the uniqueness of the direction $\nu$, hence this step is concluded.

\step{III}{Conclusion.}
We are now ready to conclude the proof. In Step~II we have already found a direction $\bar\nu$ such that
\[
\frac{A-\varphi(t)}{|A-\varphi(t)|} \to \bar\nu
\]
for $t\searrow a$. Hence, we only have to show that $\varphi'(t) \to L\bar\nu$ for $t\searrow a$: our argument will be very similar to the one of Lemma~\ref{shortonC1}. Call again $\ell$ a constant much smaller than $L\eps$, fix arbitrarily some $\eta\ll 1$, and consider the first portion of the curve $\varphi$, after $A$, of a length $\bar\delta(\ell,\eta,L)$. We claim that for any point $\varphi(t)$ in this piece of curve one has
\begin{equation}\label{ultip}
|\varphi'(t) - L\bar\nu| \leq 16\sqrt\eta\,.
\end{equation}
Once we prove this, since $\eta$ is arbitrary the proof is concluded. Assume then the existence of some $t$ as before for which~(\ref{ultip}) is not satisfied, take a point $R=\varphi(r)$, with $r>a$, such that $\segm{AR}=2\segm{AT}$, and take two more points $Q^+=\varphi(q^+)$ and $R^-=\varphi(r^-)$ with $a<q^+<r^-<r$ so that
\[
\segm{AQ^+} = \segm{R^-R} = \frac{\segm{AR}}3\,.
\]
The existence of $t$ implies that $\varphi$ is not linear between $A$ and $R$. Therefore, we can apply once again first Lemma~\ref{ifnotshort} and then Lemma~\ref{smallone} with $Q=A$, in particular getting the validity of~(\ref{prop5}). Exactly as in the proof of Lemma~\ref{shortonC1}, this implies that the direction $\nu$ of the segment $Q^+R^-$ satisfies
\[
|\nu-\bar\nu| < 16\sqrt\eta\,,
\]
hence by continuity we can find two points $Q'=\varphi(q')$ and $Q''=\varphi(q'')$ with $q^+<q'<q''<r^-$ with $|\nu'-\bar\nu|=16\sqrt\eta$, being $\nu'$ the direction of $Q'Q''$. And finally, the points $Q'$ and $Q''$ give a contradiction with~(\ref{oscest}) of Lemma~\ref{bigone}. This contradiction show the validity of~(\ref{ultip}), and then the proof is concluded.
\end{proof}

\begin{lemma}\label{nowlast}
Let $\varphi:[0,C]\to \R^2$ be an $L$-biLipschitz function, short on some interval $[a,C]$. Then, $\varphi$ is of class ${\rm C}^1$ on the interval $(a,C]$.
\end{lemma}
\begin{proof}
The ${\rm C}^1$ regularity on the open interval $(a,C)$ has been already proved in Lemma~\ref{shortonC1}, thus we only have to take care of the situation near $b=C$. Our argument will be quite similar to what already done in the proof of Lemmas~\ref{bigone} and~\ref{shortonC1}, and is divided in two steps for simplicity.
\step{I}{The directions of the segments $QB$ converge to some $\nu\in \S^1$.}
First of all, we consider the segments $QB$, where as usual $Q=\varphi(q)$ and $B=\varphi(b)$. We aim to show that the directions of the segments $QB$ converge to some $\nu\in\S^1$ when $q\to b$. Suppose that this is false and notice that, by compactness, this means that the limit points of the directions of the segments $QB$ when $q\to b$ are a connected subset of $\S^1$ made by more than one point. We can then fix a distance $\ell$ such that $\segm{PB}\gg \ell$ for every $P=\varphi(p)$ and $p<a$, and $\eta$ much smaller than the diameter of the set of the limiting directions just mentioned. Let us now pick any $q<b$, with $q$ very close to $b$ so that $\segm{QB} \leq \bar\delta(\ell,\eta,L)$; the function $\varphi$ is of course not a segment between $q$ and $b$, thus the function $\varphi^+_{qb}$ does not satisfy the inverse $L$-Lipschitz property at some pair $(p,t)$ with $p<a$ and $q<t<b$, so we can apply Lemma~\ref{smallone} and in particular we get
\begin{align}\label{punto1}
\Big|\theta\cdot \nu - \frac 1 {L^2}\Big| \leq 2\,\frac\eta{L^2}\,, && 
\frac{\H^1(\arc{QB})}{\segm{QB}} \leq 1+6\eta\,, &&
\segm{PQ} \leq \frac {q-p}L + \frac{3\eta\segm{QB}}{L^2}\,,
\end{align}
where $\theta$ and $\nu$ are the directions of the segments $PQ$ and $QB$ respectively. The very same argument can be applied to some other $Q'$ near $B$, getting another point $P'$ and 
\begin{align}\label{punto2}
\Big|\theta'\cdot \nu' - \frac 1 {L^2}\Big| \leq 2\,\frac\eta{L^2}\,, && 
\frac{\H^1(\arc{Q'B})}{\segm{Q'B}} \leq 1+6\eta\,, &&
\segm{P'Q'} \leq \frac {q'-p'}L + \frac{3\eta\segm{Q'B}}{L^2}\,,
\end{align}
being $\theta'$ and $\nu'$ the directions of $P'Q'$ and $Q'B$ respectively. Recall now that we are assuming that the limiting directions of the segments $QB$ form a nontrivial arc of $\S^1$: as a consequence, similarly as in Step~II of the proof of Lemma~\ref{rightleft}, we can select two such directions $\nu,\,\nu'$, assume by symmetry that
\begin{equation}\label{punto2.5}
\theta' \cdot \nu > \frac 1{L^2} + 10\,\frac \eta {L^2}\,,
\end{equation}
and chose two points $Q,\, Q'$ corresponding to the directions $\nu$ and $\nu'$ in such a way that $b-q' \ll b-q$, so that $\H^1(\arc{QQ'})\approx \H^1(\arc{QB})$ and then also using~(\ref{punto1}) we get
\begin{align}\label{punto3}
\segm{Q'B} \leq \frac{\segm{QQ'}}3\,, &&
\segm{QQ'} \geq \H^1(\arc{QQ'})(1-6\eta)\,, &&
|\tilde\nu - \nu|\leq \frac \eta{L^2}\,,
\end{align}
where $\tilde\nu$ is the direction of $QQ'$. Using then the estimates~(\ref{punto3}), (\ref{punto2}) and~(\ref{punto2.5}), and applying Lemma~\ref{3.4} to the points $P'$, $Q'$ and $Q$, we obtain
\[\begin{split}
\frac {q'-p'}L + \frac{\eta\segm{QQ'}}{L^2} &\geq
\frac {q'-p'}L + \frac{3\eta\segm{Q'B}}{L^2} \geq \segm{P'Q'}
\geq \segm{P'Q} + \segm{QQ'} \bigg(\theta'\cdot \tilde\nu - \frac \eta{L^2} \bigg)\\
&\geq \frac{q-p'}L + \segm{QQ'} \bigg(\theta'\cdot \nu - 2\,\frac \eta{L^2} \bigg)
\geq \frac{q-p'}L + \segm{QQ'} \bigg(\frac 1{L^2}+ 8\,\frac \eta{L^2} \bigg)\,,
\end{split}\]
which implies, again recalling~(\ref{punto3}),
\[
\frac {q'-q}L \geq \segm{QQ'} \bigg(\frac 1{L^2}+ 7\,\frac \eta{L^2} \bigg)
\geq \frac{\H^1(\arc{QQ'})}{L^2}(1-6\eta)(1+7\eta)
= \frac{q'-q}L (1+\eta-42\eta^2)\,,
\]
which is impossible as soon as $\eta$ was chosen small enough. This concludes the proof of this step.

\step{II}{The derivative $\varphi'(q)$ converge to $L\nu$.}
In order to conclude the proof, we have now to check that $\varphi'(q)\to L\nu$ when $q\to b$, being $\nu$ the direction found in Step~I. Suppose that this is not the case; then, since we already know that $\varphi'$ is continuous, with $|\varphi'|=L$, on the open interval $(a,b)$ by Lemma~\ref{shortonC1}, the set of limiting directions of the vectors $\varphi'(t)/L$ with $t\to b$ is a non-trivial arc of $\S^1$, containing of course the direction $\nu$ found in Step~I. Let us then pick a direction $\tilde\nu\neq \nu$ in the interior of this arc, satisfying
\[
|\nu-\tilde\nu| \leq \frac 14\, \min \bigg\{ \frac 1{L^2}, 1- \frac 1{L^2}\bigg\}\,,
\]
and let us select $\ell,\, \eta>0$ in such a way that $\segm{PB}\gg \ell$ for every $P=\varphi(p)$ and $p<a$, and that
\begin{equation}\label{seclas}
|\nu-\tilde\nu| > 16 \sqrt \eta\,.
\end{equation}
Let now $t<b$ be a point such that $b-t\ll \bar\delta(\ell,\eta,L)/L$, $\varphi'(t) =L \tilde\nu$ (which is possible since $\tilde\nu$ is in the interior of the arc containing all the limiting directions), and also such that
\begin{equation}\label{thilas}
\bigg|\frac{B-S}{\segm{SB}}-\nu \bigg| \leq \frac\eta{L^2} \qquad \forall\, S:\, \segm{SB} \leq 3 \segm{TB}\,;
\end{equation}
this last estimate is of course admissible thanks to Step~I. Moreover, let us fix $q<t$ so that $\segm{QB}=2\segm{TB}$. Since $\varphi$ cannot be a segment between $q$ and $b$, keeping in mind Lemma~\ref{ifnotshort} we can apply as usual Lemma~\ref{smallone} with $R=B$ and in particular we get that~(\ref{prop5}) holds true. Now, let $q'<t<q''$ be two points such that the direction $\nu'$ of the segment $Q'Q''$ satisfies
\begin{align}\label{foulas}
|\nu-\nu'| = 16 \sqrt\eta\,, && |\nu-\nu'| \leq \frac 12\, \min \bigg\{ \frac 1{L^2}, 1- \frac 1{L^2}\bigg\}\,.
\end{align}
Notice that such two points surely exist, thanks to~(\ref{seclas}) and the fact that the direction of $QB$ is very close to $\nu$ by~(\ref{thilas}). Moreover, since $\eta\ll |\nu-\nu'|$ and since the curve $\arc{QB}$ is very close to the segment $QB$ by the validity of~(\ref{prop5}), we find that $\segm{Q'Q''}\ll \delta=\segm{QB}$, hence in particular
\begin{align*}
\segm{QQ'}\geq \frac \delta 3\,, && \segm{QQ''}\geq \frac \delta 3\,, &&
\segm{Q'B}\geq \frac \delta 3\,, && \segm{Q''B}\geq \frac \delta 3\,.
\end{align*}
Finally, since~(\ref{thilas}) and~(\ref{foulas}) give that $\tilde\nu\neq \nu'$, and thus that $\varphi$ is not a segment between $Q'$ and $Q''$, and then by Lemma~\ref{ifnotshort} the inverse $L$-Lipschitz property for $\varphi^+_{q'q''}$ must fail for some $P=\varphi(p)$ with $p<a$ and some point between $Q'$ and $Q''$, we can apply Lemma~\ref{bigone} with $R=B$. However, (\ref{foulas}) is clearly in contradiction with~(\ref{oscest}), and this shows that $\varphi'(q)\to L\nu$ for $q\to b$, as desired.
\end{proof}

Let us conclude this section by simply observing that Proposition~\ref{nLI} is an immediate consequence of Lemma~\ref{shortonC1}, Lemma~\ref{rightleft} and Lemma~\ref{nowlast} (and the symmetric counterparts of the last two): indeed, the internal regularity is given by Lemma~\ref{shortonC1}, while the regularity up the boundary is achieved applying Lemma~\ref{rightleft} or Lemma~\ref{nowlast} for points of the boundary which are in $(0,1)$ or in $\{0,1\}$.

\section{Proof of Theorem~\ref{main}\label{sect4}}

This section is devoted to show Theorem~\ref{main}. The idea is simply to put together Proposition~\ref{step1} and Proposition~\ref{nLI}; with the first result, one gets a bi-Lipschitz function which is piecewise linear on most of $[0,1]$, and then the second result allows to modify the function on the small regions which are left out.

\begin{proof}[Proof of Theorem~\ref{main}]
We take a small constant $\xi=\xi(L,\eps)$, to be specified at the end. For the sake of simplicity, we divide the proof in a few steps.
\step{I}{The function $\varphi_1$ from Proposition~\ref{step1}.}
We start by applying Proposition~\ref{step1} to $\varphi$, thus getting an $(L+\xi)$-biLipschitz function $\varphi_1:[0,1]\to\R^2$, which satisfies~(\ref{claimmain}) and which is piecewise linear on a finite union $A$ of closed intervals which cover a portion of length at least $1-\xi$ of the whole $[0,1]$. Let us then write
\[
[0,1]\setminus A = \bigcup\nolimits_{i=1}^N J_i\,,
\]
where the $J_i$'s are a finite number of open intervals, satisfying
\[
\sum_{i=1}^N |J_i| = 1 - |A| \leq \xi\,.
\]

\step{II}{The function $\varphi_2$, which ``goes fast'' near the intervals $J_i$.}
In this step, we make a simple modification of $\varphi_1$, in order to be able to apply Proposition~\ref{nLI} later. To do so, we recall that the function $\varphi_1$ is finitely piecewise linear on a subset $A$ of $[0,1]$; hence, we can define a very small length $\ell>0$ such that any interval (contained in $A$) where $\varphi_1$ is linear, is much longer than $\ell$. This is of course possible since such intervals are finitely many. In particular, the distance between any two consecutive ``bad'' intervals $J_i$ and $J_{i+1}$ is always much larger than $\ell$. Up to further decrease $\ell$, we can also assume that
\begin{equation}\label{smallell}
2\ell N < \xi\,.
\end{equation}
Let us now define $A^-$ the subset of $[0,1]$ made by all the points of $A$ which have distance from $[0,1]\setminus A$ smaller than $\ell$: by construction, $A^-$ is a union of either $2N-2$, or $2N-1$, or $2N$ small subintervals of $A$, on each of which $\varphi_1$ is linear by construction: the exact number depends on whether $0$ and/or $1$ belong to $A$ or not. Let us now consider the function $\tau:[0,1]\to [0,C]$ identified by
\begin{align*}
\tau(0)=0\,, && \tau'(x) = 1 \quad \forall\, x\notin A^-\,, && \tau'(x) = \frac {|\varphi_1'(x)|}{L+\xi} \quad \forall\, x\in A^-\,.
\end{align*}
It is immediate from the definition to observe that $\tau'(x)\leq 1$ for every $x$, and that
\[
1 - 2N\ell\leq C \leq 1\,,
\]
which also by~(\ref{smallell}) implies
\[
x-\xi \leq \tau(x) \leq x \qquad \forall\, 0\leq x\leq 1\,,
\]
so that in particular
\begin{equation}\label{star}
1 - \xi \leq \tau(1) = C\,.
\end{equation}
Finally, we define the function $\varphi_2: [0,C]\to \R^2$ as
\[
\varphi_2(x) = \varphi_1\big(\tau^{-1}(x)\big)\,.
\]
Since $\tau'\leq 1$, as already observed the inverse Lipschitz property for $\varphi_2$ works better as than for $\varphi_1$, hence $\varphi_2$ satisfies the $L$ inverse Lipschitz property because so does $\varphi_1$. On the other hand, $|\varphi_2'(x)| = |\varphi_1'(\tau^{-1}(x))|\leq L$ as soon as $\tau^{-1}(x)\notin A^-$\,, while otherwise $|\varphi_2'(x)| =L+\xi$. As a consequence, also the $(L+\xi)$-Lipschitz property for $\varphi_2$ follows. Summarizing, the function $\varphi_2:[0,C]\to \R^2$ is an $(L+\xi)$-biLipschitz function, which is finitely piecewise linear on the whole $[0,C]$ except on $N$ intervals $\widetilde J_i$, where each $\widetilde J_i$ is simply $\tau(J_i)$, and so $|\widetilde J_i|=|J_i|$. Observe that, by construction, the function $\varphi_2$ is linear, with $|\varphi_2'|=L+\xi$, for a short while before and after each of the intervals $\widetilde J_i$, except of course before $\widetilde J_1$ if $0\in \widetilde J_1$, and after $\widetilde J_N$ if $C\in \widetilde J_N$. We conclude this step with a couple of observations. First of all, by construction, $\varphi_2$ is finitely piecewise linear on a subset $\widetilde A$ of $[0,C]$ which satisfies
\begin{equation}\label{reclen}
C - |\widetilde A| = C - |\tau(A)| = \big|\tau \big( [0,1] \setminus A\big)\big| 
=| [0,1] \setminus A| = 1 - |A| \leq \xi\,.
\end{equation}
And moreover, for every $x\in [0,C]$, we have
\begin{equation}\label{rat}
\big|\varphi_2(\tau(x)) - \varphi(x)\big| = \big|\varphi_1(x)- \varphi(x)\big| \leq \xi
\end{equation}
recalling that~(\ref{claimmain}) holds for $\varphi_1$.

\step{III}{The function $\varphi_3$, which is short on each $\widetilde J_i$.}
We can now further modify the function $\varphi_2$. We simply apply Proposition~\ref{nLI} (with $L+\xi$ in place of $L$) on each of the intervals $\widetilde J_i$; more precisely, we first apply the proposition to the interval $\widetilde J_1$, finding a map $\varphi_3^1:[0,C_1]\to \R^2$ which is $(L+\xi)$-biLipschitz and satisfies~(\ref{propprel}). Then, we apply again Proposition~\ref{nLI} to the interval $\widetilde J_2^1$, which is simply the interval $\widetilde J_2$ translated of a distance $C-C_1$, so that $\varphi_3^1$ on $\widetilde J_2^1$ coincides with $\varphi_2$ on $\widetilde J_2$. Going on with the obvious recursion, after $N$ steps we have finally defined the function $\varphi_3: [0,C']\to \R^2$, which is $(L+\xi)$-biLipschitz, and which by construction is finitely piecewise linear on some subset $A'$ of $[0,C']$ satisfying
\[
C' - |A'| = C - |\widetilde A| \leq \xi\,,
\]
recalling~(\ref{reclen}). We can say also something more precise: $[0,C']\setminus A'$ is the union of $N$ intervals $J_i'$, and on the closure of each of them the function $\varphi_3$ is of class ${\rm C}^1$. In addition, since the function has been changed only on the intervals $\widetilde J_i$, and doing so those intervals have been shrinked, there exists a function $\tilde\tau:[0,C]\to [0,C']$ such that
\begin{align*}
\tilde\tau(\widetilde A) = A'\,, && \tilde\tau'(x) = 1\quad \forall\, x\in \widetilde A\,,
\end{align*}
and by~(\ref{rat}) we get
\begin{equation}\label{ratt}
\big| \varphi_3\big(\tilde\tau(\tau(x))\big) - \varphi(x)\big| =
\big| \varphi_2(\tau(x)) - \varphi(x)\big| \leq \xi \qquad \forall\, x\in \big(\tilde\tau \circ \tau\big)^{-1} (A')\,.
\end{equation}

\step{IV}{The function $\varphi_4$, which is finitely piecewise linear.}
From the previous steps, we have now a function $\varphi_3:[0,C']\to \R^2$ which is finitely piecewise linear on almost the whole $[0,C']$, and which is ${\rm C}^1$ on the closure of each of the $N$ intervals $J_i'$ where it is not already piecewise linear.\par

As pointed out in Remark~\ref{C1}, it is elementary how to modify a ${\rm C}^1$ function into a finitely piecewise linear one, up to increase the biLipschitz constant of an arbitrarily small constant. Applying this argument $N$ times, to each of the intervals $J_i'$, we get then a finitely piecewise linear function $\varphi_4:[0,C']\to \R^2$, which is $(L+2\xi)$-biLipschitz and which of course satisfies $\varphi_4(0)=\varphi(0)$ and $\varphi_4(C')=\varphi(1)$. And moreover, since $\varphi_4=\varphi_3$ on $A'$, then the estimate~(\ref{ratt}) holds also with $\varphi_4$ in place of $\varphi_3$.

\step{V}{The ``final'' function $\varphi_\eps$.}
We are finally in position to conclude the proof of our Theorem. The function $\varphi_4$ was already almost perfect, its only flaw being that it is defined on the interval $[0,C']$ instead than on $[0,1]$. Nevertheless, we can easily observe that
\begin{equation}\label{esob}
1-2\xi \leq C' \leq 1\,.
\end{equation}
Indeed, the fact that $C'\leq 1$ is obvious, since all our modifications of the map $\varphi$ either left the domain unchanged or shrinked it. On the other hand, recalling~(\ref{reclen}) and~(\ref{star}), we have
\begin{equation}\label{abt}
C' \geq |A'| = |\widetilde A| \geq C-\xi \geq 1-2\xi\,,
\end{equation}
so the validity of~(\ref{esob}) is established.\par

The function $\varphi_\eps$ will then be simply a reparameterization of $\varphi_4$, precisely we set $\varphi_\eps : [0,1]\to\R^2$ as
\[
\varphi_\eps (x) = \varphi_4 ( C' x)\,.
\]
The function $\varphi_\eps$ is then finitely piecewise linear by construction, of course it satisfies $\varphi_\eps(0)=\varphi(0)$ and $\varphi_\eps(1)=\varphi(1)$, and it is at most $(L+\xi)(1+3\xi)$-biLipschitz, hence in particular $(L+\eps)$-biLipschitz if $\xi(L,\eps)$ is suitably small.\par

To conclude, we have then just to check that $\|\varphi_\eps-\varphi\|_{L^\infty}\leq \eps$. To do so, let us take a generic $z \in [0,1]$; first of all, we can find $x\in A'' = A'/C'$ such that, also by~(\ref{abt}),
\[
|z-x| \leq 1 - |A''| = 1 - \frac{|A'|}{C'} \leq 2\xi\,.
\]
Then, we define $y=(\tilde\tau\circ\tau)^{-1}(C'x)$ and, since
\[
y-2\xi \leq y - (1-C') \leq \tilde\tau(\tau(y))=C' x \leq y \,,
\]
also recalling that~(\ref{ratt}) holds also with $\varphi_4$ in place of $\varphi_3$ we deduce
\begin{align*}
|y-x| \leq |y-C'x| + |C'-1| \leq 4\xi\,, &&
\big|\varphi_\eps(x) -\varphi(y)\big|
=\big|\varphi_4(C'x) - \varphi(y)\big|
\leq \xi\,.
\end{align*}
Recalling that $\varphi$ is $L$-biLipschitz while $\varphi_\eps$ is $(L+\eps)$-biLipschitz, the above estimates give us
\[\begin{split}
|\varphi_\eps(z) - \varphi(z)| &\leq |\varphi_\eps(z) - \varphi_\eps(x)| + |\varphi_\eps(x) - \varphi(y) |  + |\varphi(y)-\varphi(z)|\\
&\leq (L+\eps)|z -x| + \xi  + L|y-z|
\leq (L+\eps) 2\xi + \xi + 6 L \xi\,.
\end{split}\]
Since $z\in [0,1]$ was generic, and since the last quantity is smaller than $\eps$ as soon as $\xi=\xi(L,\eps)$ is small enough, the $L^\infty$ estimate has been established, and the proof is concluded.
\end{proof}

A straightforward consequence of our construction is the following.

\begin{corollary}\label{cormain}
Assume that $\varphi:[0,1]\to\R^2$ is an $L$-biLipschitz function, linear on $[0,a]$ and on $[1-a,1]$ for some $a\ll 1$. Fix two quantities $0<a'<a$ and $\eps>0$. There exists an $(L+\eps)$-biLipschitz function $\varphi_\eps:[0,1]\to\R^2$ such that~(\ref{claimmain}) holds, $\varphi_\eps$ is finitely piecewise linear on $[0,1]$, and $\varphi_\eps$ coincides with $\varphi$ on the intervals $[0,a']$ and $[1-a',1]$.
\end{corollary}
\begin{proof}
To obtain the required function $\varphi_\eps$, it is enough to check the proof of Theorem~\ref{main} and to modify it only very slightly.\par

Indeed, the first step of that proof simply consists in taking a function $\varphi_1$ given by Proposition~\ref{step1}. On the other hand, from a quick look at the proof of Proposition~\ref{step1}, it is obvious that $\varphi_1$ coincides with $\varphi$ in all the intervals $[m/N,(m+1)/N]$ where $\varphi$ is linear. As a consequence, we can select any $\delta>0$ and, up to take $N$ much bigger than $1/a$ and $1/\delta$ in the proof of Proposition~\ref{step1}, we get a function $\varphi_1$ which coincides with $\varphi$ on $[0,a-\delta]$ and on $[1-(a-\delta),1]$, and in particular these two intervals are contained in the set $A$ of Step~I.\par

In the second step of the proof of Theorem~\ref{main}, we just modified $\varphi_1$ in order to make it faster near the end of the good intervals, getting then a new function $\varphi_2:[0,C]\to\R^2$. Up to take $\ell$ smaller than $\delta$ there, we can do the same construction and we have then that $\varphi_2=\varphi$ on $[0,a-2\delta]$ and that $\varphi_2(x+C-1)=\varphi(x)$ for each $x\in [1-(a-2\delta),1]$.\par

In the third and fourth step of the proof, we defined a function $\varphi_4:[0,C']\to\R^2$, modifying $\varphi_2$ only in the bad intervals. Again, we can do exactly the same thing now, and we still have that $\varphi_4$ coincides with $\varphi$ on $[0,a-2\delta]$ and that $\varphi_4(x+C'-1)=\varphi(x)$ for each $x\in [1-(a-2\delta),1]$.\par

Finally, in the last step we defined the approximating function $\varphi_\eps$, which was obtained simply by ``changing the velocity'' of $\varphi_4$, namely, we set $\varphi_\eps(x)=\varphi_4(C'x)$. This time, we cannot do the same, otherwise we lose the information that $\varphi$ and $\varphi_\eps$ coincide near $0$ and $1$; nevertheless, it is clear that a solution is just to define
\[
\varphi_\eps(x) = \left\{\begin{array}{ll}
\varphi_4(x)  &\hbox{for $0\leq x \leq a$}\,,\\[5pt]
\varphi_4\bigg(\bal a+\frac{C'-2a}{1-2a}\,(x-a)\eal\bigg) \qquad &\hbox{for $a\leq x \leq 1-a$}\,,\\[10pt]
\varphi_4(x+C'-1)  &\hbox{for $1-a\leq x \leq 1$}\,.
\end{array}\right.
\]
Indeed, with this definition, the very same arguments as in Step~V of the proof of Theorem~\ref{main} still ensure that $\varphi_\eps$ is $(L+\eps)$-biLipschitz and that~(\ref{claimmain}) holds; moreover, $\varphi_\eps$ is finitely piecewise linear by definition. Finally, $\varphi_\eps$ coincides with $\varphi$ on $[0,a-2\delta]$ and on $[1-(a-2\delta),1]$ thus, provided that we have chosen $\delta$ smaller than $(a-a')/2$, the proof is concluded.
\end{proof}

\section{generalization to $\S^1$\label{sect5}}

In this section we generalize Theorem~\ref{main} in order to consider the case of a map defined on $\S^1$, instead than on $[0,1]$. To do so, we need the following standard definitions.

\begin{definition}
Let $\pi:\R\to \S^1$ be the map $\pi(t) = t \, ({\rm mod}\ 2\pi)$, let $\varphi:\S^1\to\R^2$ be any function, and let $a<b$ be any two real numbers such that $b-a<2\pi$. We denote by $\varphi^{ab}:[a,b]\to\R^2$ the function defined by $\varphi^{ab}(t) = \varphi(\pi(t))$ for any $a\leq t\leq b$. We say that the function $\varphi$ is \emph{finitely piecewise linear} if so is the function $\varphi^{ab}$ for any choice of $a$ and $b$.
\end{definition}

The goal of this last section is to prove is the following statement.
\begin{theorem}\label{main2}
Let $\varphi:\S^1\to\R^2$ be an $L$-biLipschitz function, and $\eps>0$. Then, there exists a finitely piecewise linear, $(L+\eps)$-biLipschitz function $\varphi_\eps:\S^1\to\R^2$ such that $\|\varphi-\varphi_\eps\|_{L^\infty}\leq \eps$.
\end{theorem}
\begin{proof}
First of all, let us fix a small positive quantity $\eps'$, to be specified later, and let us also fix a small $\theta>0$ such that
\begin{equation}\label{deftheta}
1-\eps' \leq \frac{2\sin(\theta/2)}\theta \leq 1\,.
\end{equation}
We divide the proof in few steps for clarity.
\step{I}{The Lebesgue points for $\varphi$ and the function $\varphi_1$.}
Let $p\in \S^1$ be a Lebesgue point for $\varphi$, that is, for any $a<z<b$ such that $p=\pi(z)\in \big(\pi(a),\pi(b)\big)$ the point $z$ is a Lebesgue point for $\varphi^{ab}$. In particular, let us choose $z\in\R$ so that $p=\pi(z)$, and let us set for a moment $a=z-\theta/2$ and $b=z+\theta/2$. Notice that, by~(\ref{deftheta}), for any $a<x<y<b$ one has
\begin{equation}\label{stes}
(1-\eps') |y-x| \leq |\pi(y)-\pi(x)|\leq |y-x|\,.
\end{equation}
Let us then concentrate ourselves on the function $\varphi^{ab}$, which is easily $L_1$-biLipschitz with $L_1=L(1-\eps')^{-1}$ thanks to~(\ref{stes}). The point $z$ is a Lebesgue point for $\varphi^{ab}$, hence we can apply to it Lemma~\ref{Lebesgue}, using $[a,b]$ in place of $[0,1]$ and with constant $\eps'$ in place of $\eps$. We get then a constant $\bar \ell$ and the sets $I_\ell(z)$ for $\ell\leq\bar\ell$. We can arbitrarily select $\ell\ll \theta$ and two points $s<z<t$ in $I_\ell(z)$, and the lemma ensures us that the function $\psi:[a,b]\to\R^2$ defined as $\psi=\varphi^{ab}_{st}$ is $L_2=(L_1+\eps')$-biLipschitz and satisfies $\|\psi-\varphi^{ab}\|_{L^\infty}\leq \eps'$. We can then define the function $\tilde\varphi:\S^1\to\R^2$ as $\tilde\varphi=\varphi$ outside the arc $\pi(a)\pi(b)$, and $\tilde\varphi(q) = \psi(\pi^{-1}(q))$ inside. By construction we have that $\|\varphi-\tilde\varphi\|_{L^\infty}\leq \eps'$, and moreover $\tilde\varphi$ is $L_3$-biLipschitz with $L_3=L_2(1-\eps')^{-1}$: this can be obtained arguing exactly as in Lemma~\ref{Lebesgue}, and keeping in mind that $\ell\ll\theta$.\par

We want now to use a similar argument with many Lebesgue points, instead of just one. To do so, let us select finitely many Lebesgue points $p_1,\, p_2,\, \dots \, , \, p_M$ in $\S^1$, such that every arc $\arc{p_ip_{i+1}}$ in $\S^1$ has length less than $\theta$, and it does not contain other points $p_j$ (of course, as usual we denote $p_{M+1}\equiv p_1$). For each of these points, say $p_i$, we can repeat the argument above, selecting $\ell$ much smaller than the minimal distance between two of the points $p_j$, and finding a function $\tilde\varphi_i$ which is a segment between $\tilde\varphi_i(s_i)=\varphi(s_i)$ and $\tilde\varphi_i(t_i)=\varphi(t_i)$. Hence, each function $\tilde\varphi_i$ coincides with $\varphi$ in all $\S^1$ except a small arc $\arc{s_it_i}$ around $p_i$, and these arcs are all well disjoint. We can then define $\varphi_1:\S^1\to\R^2$ as the function which coincides with $\tilde\varphi_i$ in every $\arc{s_it_i}$, and with $\varphi$ otherwise. Arguing as in Section~\ref{sect2}, in particular keeping in mind that each length $\ell$ has been chosen much smaller than the distance between different points $p_j$, we obtain immediately that $\varphi_1$ is $L_4$-biLipschitz, with $L_4=L_3+\eps'$. Moreover, by construction $\varphi_1$ is linear on each arc $(s_i,t_i)$, and $\|\varphi_1-\varphi\|_{L^\infty}\leq \eps'$.

\step{II}{Modification in each arc $\arc{p_ip_{i+1}}$ and the function $\varphi_\eps$.}
We now restrict our attention to the function $\varphi_1$ on the arc $\arc{p_ip_{i+1}}$. This is an $L_4$-biLipschitz function, and its image is a segment at the beginning and at the end, that is, from $p_i$ to $t_i$ and from $s_{i+1}$ to $p_{i+1}$. Let us then take two real numbers $a<b$ with $b-a<2\pi$ and $\pi(a)=p_i$, $\pi(b)=p_{i+1}$, and let us call $\psi=\varphi_1^{ab}$. Moreover, let $a^+,\, b^- \in (a,b)$ be such that $\pi(a^+)=t_i$ and $\pi(b^-)=s_{i+1}$: hence, the function $\psi : [a,b]\to\R^2$ is an $L_5$-biLipschitz function with $L_5=L_4(1-\eps')^{-1}$, again by~(\ref{stes}), and it is linear in $[a,a^+]$ and in $[b^-,b]$. We can then apply Corollary~\ref{cormain}, so we get a piecewise linear function $\tilde\psi:[a,b]\to\R^2$, biLipschitz with constant $L_6=L_5+\eps'$, such that $\|\tilde\psi-\psi\|_{L^\infty}\leq \delta$, and coinciding with $\psi$ (thus, linear) on the two intervals $[a,a']$ and $[b',b]$, for two points $a<a'<a^+$ and $b^-<b'<b$, and for a suitable constant $\delta>0$ to be specified later.\par

We define then $\varphi_2^i:\S^1\to\R^2$ as the function which coincides with $\varphi_1$ out of the arc $\arc{p_ip_{i+1}}$, and with $\tilde\psi\circ \pi^{-1}$ inside the arc. Notice that the function $\varphi_2^i$ is piecewise linear in the arc $\arc{p_ip_{i+1}}$, and in particular it is linear and coincides with $\varphi_1$ in the small arcs $\arc{p_it_i'}$ and $\arc{s_{i+1}'p_{i+1}}$, where the points $t_i'\in \arc{p_it_i}$ and $s_{i+1}'\in\arc{s_{i+1}p_{i+1}}$ are $\pi^{-1}(a')$ and $\pi^{-1}(b')$ respectively. Moreover, we also have $\|\varphi_2^i-\varphi_1\|_{L^\infty}=\|\tilde\psi-\psi\|_{L^\infty}\leq \delta$.\par

Finally, we repeat the same construction for each $1\leq i\leq M$, and we define the final function $\varphi_\eps$ as the function coinciding with $\varphi_2^i$ in each arc $\arc{p_ip_{i+1}}$, so that $\|\varphi_\eps-\varphi_1\|_{L^\infty}\leq \delta$.

\step{III}{Conclusion.}
It is now only left to check that the function $\varphi_\eps$ satisfies the requirement of the Theorem. By construction we have that $\varphi_\eps$ is finitely piecewise linear, and moreover
\[
\|\varphi_\eps-\varphi\|_{L^\infty}\leq \|\varphi_\eps-\varphi_1\|_{L^\infty}+\|\varphi_1-\varphi\|_{L^\infty}\leq \delta+\eps'\,,
\]
so we have $\|\varphi_\eps-\varphi\|_{L^\infty}\leq \eps$ as soon as we have chosen $\delta$ and $\eps'$ small enough. To conclude, we only have to check that $\varphi_\eps$ is $(L+\eps)$-biLipschitz.\par

To do so, let us take two points $x,\, y\in \S^1$. Suppose first that they belong to a same arc $\arc{p_ip_{i+1}}$. Then, by construction we have, setting $L_7=L_6(1-\eps')^{-1}$,
\begin{equation}\label{last1}\begin{split}
|\varphi_\eps(y)-\varphi_\eps(x)|
&=|\varphi_2^i(y)-\varphi_2^i(x)|
=\big|\tilde\psi(\pi^{-1}(y))-\tilde\psi(\pi^{-1}(x))\big|
\leq L_6 |\pi^{-1}(y)-\pi^{-1}(x)|\\
&\leq L_6(1-\eps')^{-1} |y-x|
=L_7|y-x|\,,
\end{split}\end{equation}
since $\tilde\psi$ is $L_6$-biLipschitz by Step~II and again by~(\ref{stes}).

Suppose now, instead, that $x$ and $y$ belong to two different arcs, in particular let us take $x\in \arc{p_ip_{i+1}}$ and $y\in \arc{p_jp_{j+1}}$. We can divide this case in two subcases, namely, if the equality $\varphi_\eps=\varphi_1$ holds at both $x$ and $y$, or not. If $\varphi_\eps(x)=\varphi_1(x)$ and $\varphi_\eps(y)=\varphi_1(y)$, then since $\varphi_1$ is $L_4$-biLipschitz we have that
\begin{equation}\label{last2}
|\varphi_\eps(y)-\varphi_\eps(x)|=
|\varphi_1(y)-\varphi_1(x)| \leq L_4 |y-x|\,.
\end{equation}
Finally, assume (by symmetry) that $\varphi_\eps(x)\neq \varphi_1(x)$. By construction, this implies that $x\in\arc{t_i's_{i+1}'}$; since $y\notin \arc{p_ip_{i+1}}$, we derive that $|y-x|\geq \eta$, where we define
\[
\eta = \min \Big\{ |d-c|:\, \exists\ 1\leq h \leq M,\, d\notin \arc{p_hp_{h+1}},\, c \in \arc{t_h's_{h+1}'}\Big\}\,.
\]
Notice that $\eta$ is strictly positive, since the arcs $\arc{p_hp_{h+1}}$ are only finitely many. Moreover, notice that we are free to decide $\delta$ depending on $\eta$, thanks to the construction of Step~II. As a consequence, recalling that $\varphi_1$ is $L_4$-biLipschitz and that $\|\varphi_\eps-\varphi_1\|_{L^\infty}\leq \delta$, we have for this last case
\begin{equation}\label{last3}
\frac{|\varphi_\eps(y)-\varphi_\eps(x)|}{|y-x|} \leq \frac{|\varphi_1(y)-\varphi_1(x)|}{|y-x|} + \frac{2\delta}{\eta}
\leq L_4 + \frac{2\delta}{\eta}\leq L_4+\eps'\,,
\end{equation}
where the last inequality is true as soon as $\delta$ has been chosen small enough.\par

We are then in position to conclude: it is straightforward to check that all the constants $L_j$ for $1\leq j\leq 7$ converge to $L$ when $\eps'$ go to $0$, then the estimates~(\ref{last1}), (\ref{last2}) and~(\ref{last3}) give that $\varphi_\eps$ is $(L+\eps)$-biLipschitz as soon as $\eps'$ is small enough.
\end{proof}

\end{document}